\documentclass{amsart}

\usepackage{epsfig}
\usepackage{subfigure}
\usepackage{calc}
\usepackage{amsaddr}
\usepackage{amssymb}
\usepackage{amstext}
\usepackage{amsmath}
\usepackage{amsthm}
\usepackage{pslatex}
\usepackage{hyperref}
\hypersetup{colorlinks=true, linkcolor=black,citecolor=black}

\def\R{\mathbb{R}}

\def\S{\mathbb{S}}

\newtheorem{remark}{Remark}

\newtheorem{prop}{Proposition}

\graphicspath{{./Figures/}}

\begin{document}

\title{Diffeomorphic registration of discrete geometric distributions}


\author{Hsi-Wei Hsieh}
\address{Center of Imaging Sciences, Johns Hopkins University, Baltimore, USA}
\email{hhsieh@cis.jhu.edu}

\author{Nicolas Charon}
\address{Center of Imaging Sciences, Johns Hopkins University, Baltimore, USA}
\email{charon@cis.jhu.edu}


\begin{abstract}
    This paper proposes a new framework and algorithms to address the problem of diffeomorphic registration on a general class of geometric objects that can be described as discrete distributions of local direction vectors. It builds on both the large deformation diffeomorphic metric mapping (LDDMM) model and the concept of oriented varifolds introduced in previous works like \cite{Charon2017}. Unlike previous approaches in which varifold representations are only used as surrogates to define and evaluate fidelity terms, the specificity of this paper is to derive direct deformation models and corresponding matching algorithms for discrete varifolds. We show that it gives on the one hand an alternative numerical setting for curve and surface matching but that it can also handle efficiently more general shape structures, including multi-directional objects or multi-modal images represented as distributions of unit gradient vectors.
\end{abstract}

\maketitle

\section{Introduction}
\label{sec:introduction}
\subsection*{Background} Statistical shape analysis is now regarded across the board as an important area of applied mathematics as it has been and still is the source of quantities of theoretical works as well as applications to domains like computational anatomy, computer vision or robotics. Broadly speaking, one of its central aim is to provide quantitative/computational tools to analyze the variability of geometric structures in order to perform different tasks such as shape comparison or classification. 

There are several specific difficulties in tackling such problems in the case of datasets involving geometric shapes. A fundamental one is the issue of defining and computing metrics on shape spaces. A now quite standard approach which was pioneered by Grenander in \cite{Grenander1993} is to compare shapes through distances based on deformation groups equipped with right-invariant metrics together with a left group action defined on the set of shapes. In this framework, the induced distance is typically obtained by solving a registration problem i.e by finding an optimal deformation mapping one object on the other one. It is thus ultimately determined by the deformation group and its metric for which many models have been proposed. In this paper, we will focus on the Large Deformation Diffeomorphic Metric Mapping (LDDMM) of \cite{Beg2005} in which diffeomorphic transformations are generated as flows of time-dependent velocity fields.

Despite the versatility of such models, one of the other common difficulty in shape analysis is the multiple forms or modalities that shapes may take. Looking only at the applications in the field of computational anatomy, if early works have mostly considered shapes given by medical images \cite{Rueckert1999,Beg2005} or manually extracted landmarks \cite{Joshi2000}, the variety of geometric structures at hand has considerably increased since then, whether shapes are images acquired through multiple modalities (MRI, CT...) \cite{Avants2008}, vector or tensor fields as in Diffusion Tensor Images \cite{Cao2005}, fields of orientation distribution functions \cite{Du2012} or delineated objects like point clouds, curves \cite{Glaunes2006}, surfaces \cite{Glaunes2008}, fiber bundles \cite{Durrleman4}...   

The intent of this paper is to make a modest step toward one possible generalized setting that could encompass a rich class of shapes including many of the previous cases within a common representation and eventually lead to a common LDDMM matching framework. Our starting point is the set of works on curve and surface registration based on geometric distributions like measures, currents or varifolds \cite{Glaunes2004,Glaunes2008,Charon2}. In the recent article \cite{Charon2017} for instance, an oriented curve/surface is interpreted as a directional distribution (known as oriented varifold) of its oriented tangent/normal vectors, which results in simple fidelity terms used in combination with LDDMM to formulate and solve inexact matching problems. Yet all those works so far have restricted the role of distributions' representations to intermediates for the computation of guiding terms in registration algorithms; the underlying deformation model and registration problem remains defined over point sets with meshes.

The stance we take here is to instead introduce group actions and formulate the diffeomorphic matching problem directly in spaces of geometric distributions. In this particular work, we will restrict the analysis to objects in 2D and 3D and focus on the simpler subspace of discrete distributions, i.e that write as finite sums of Dirac varifold masses: Figure \ref{fig:discr_shape_var} gives a few examples of objects naturally represented in this form. We shall consider different models of group actions and derive the corresponding optimal control problems, optimality conditions (Section \ref{sec:optimal_control}) and registration algorithms (Section \ref{sec:matching_algo}). This provides, on the one hand, an alternative (and theoretically equivalent) numerical framework to \cite{Charon2017} for curve and surface matching using currents, oriented or unoriented varifolds. But the main contribution of our proposed model is that it extends LDDMM registration to the more general class of objects representable by discrete varifolds. In Section \ref{sec:results}, we will show several examples of synthetic data besides curves or surfaces that can be treated as such, including cases like multi-directional objects or contrast-invariant images.          

\subsection*{Related works.} A few past works share some close connections with the present paper. For instance, \cite{Cao2005} develops an approach for registration of vector fields also within the LDDMM setting. The discrete distributions we consider here are however distinct from vector fields as they should rather be interpreted as unlabelled particles at some locations in space with orientation vectors attached (and with possibly varying number of orientation vectors at a single position) as opposed to a field of vectors defined on a fixed grid. In particular, our approach will be naturally framed in the Lagrangian setting as opposed to the Eulerian formulation of \cite{Cao2005}. The geodesic equations for the pushforward group action that are derived in Section \ref{sec:optimal_control} can be also related to the framework of \cite{Sommer2013} where deformations between images are estimated by matching higher-order information like the Jacobian of the diffeomorphism at given points using higher-order similarity measures with a specific form. These are defined through labelled sets of control points though and need to be first extracted from the images, which is again different and arguably less flexible than the method we introduce here.

\section{Shapes and discrete varifolds}
\label{sec:discrete_varifolds}
The idea of representing shapes as distributions goes back to the many works within the field of Geometric measure theory. Those concepts have later been of great interest in the construction of simple and numerically tractable metrics between curves or surfaces for registration problems: the works of \cite{Glaunes2006,Glaunes2008,Durrleman4,Charon2} are a few examples. The framework of oriented varifolds recently exploited in \cite{Charon2017} was shown to encompass all those notions into a general representation and provide a wide range of metrics on the spaces of embedded curves or surfaces. We give a brief summary of the latter work below.  

In the rest of the paper, we will call an oriented varifold or, to abbreviate, a \textit{varifold} in $\R^n$ (we shall here consider the cases $n=2$ or $n=3$) a distribution on the product $\R^n \times \S^{n-1}$. In other words, a varifold $\mu$ is by definition a linear form over a certain space $W$ of smooth functions on $\R^n \times \S^{n-1}$, which evaluation we shall write as $\mu(\omega)$ for any test function $\omega \in W$. In all what follows, we shall restrict our focus to 'discrete' shapes and varifolds, leaving aside the analysis of the corresponding continuous models. By discrete varifold, we mean specifically that $\mu$ writes as a finite combination of Dirac masses $\mu=\sum_{i=1}^{P} r_i \delta_{(x_i,d_i)}$ with $r_i>0$, $(x_i,d_i) \in \R^n \times \S^{n-1}$ for all $i$, in which case $\mu(\omega) = \sum_{i=1}^{P} r_i \omega(x_i,d_i)$ for all $\omega$. Such a $\mu$ can be thought as a set of unit direction vectors $d_i$ located at positions $x_i$ with weights (or masses) equal to the $r_i$'s. We assume by convention that the $(x_i,d_i)$ are distinct, but not necessarily that all the positions $x_i$ are: in other words, in our model, there can be more than a single direction vector attached to each position. In the rest of the paper, we will denote by $\mathcal{D}$ the set of all discrete varifolds. Note that in this representation and unlike the cases of landmarks and vector fields, the particles are unlabelled i.e the varifold $\mu$ is invariant to any permutation of the $(x_i,d_i)$. One particular subset of interest that we shall denote $\mathring{\mathcal{D}} \subset \mathcal{D}$ is the space of discrete varifolds with distinct positions $x_i$ (or equivalently, the discrete varifolds that carry a single direction vector per point position).     

\begin{figure}
    \centering
    \begin{tabular}{ccc}
    \includegraphics[width=4cm]{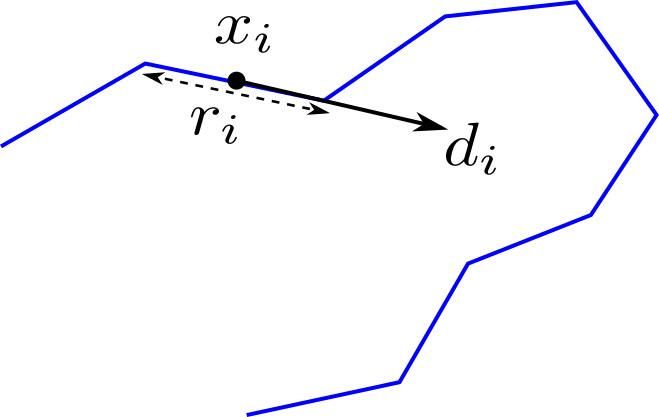} & & \includegraphics[width=4.5cm]{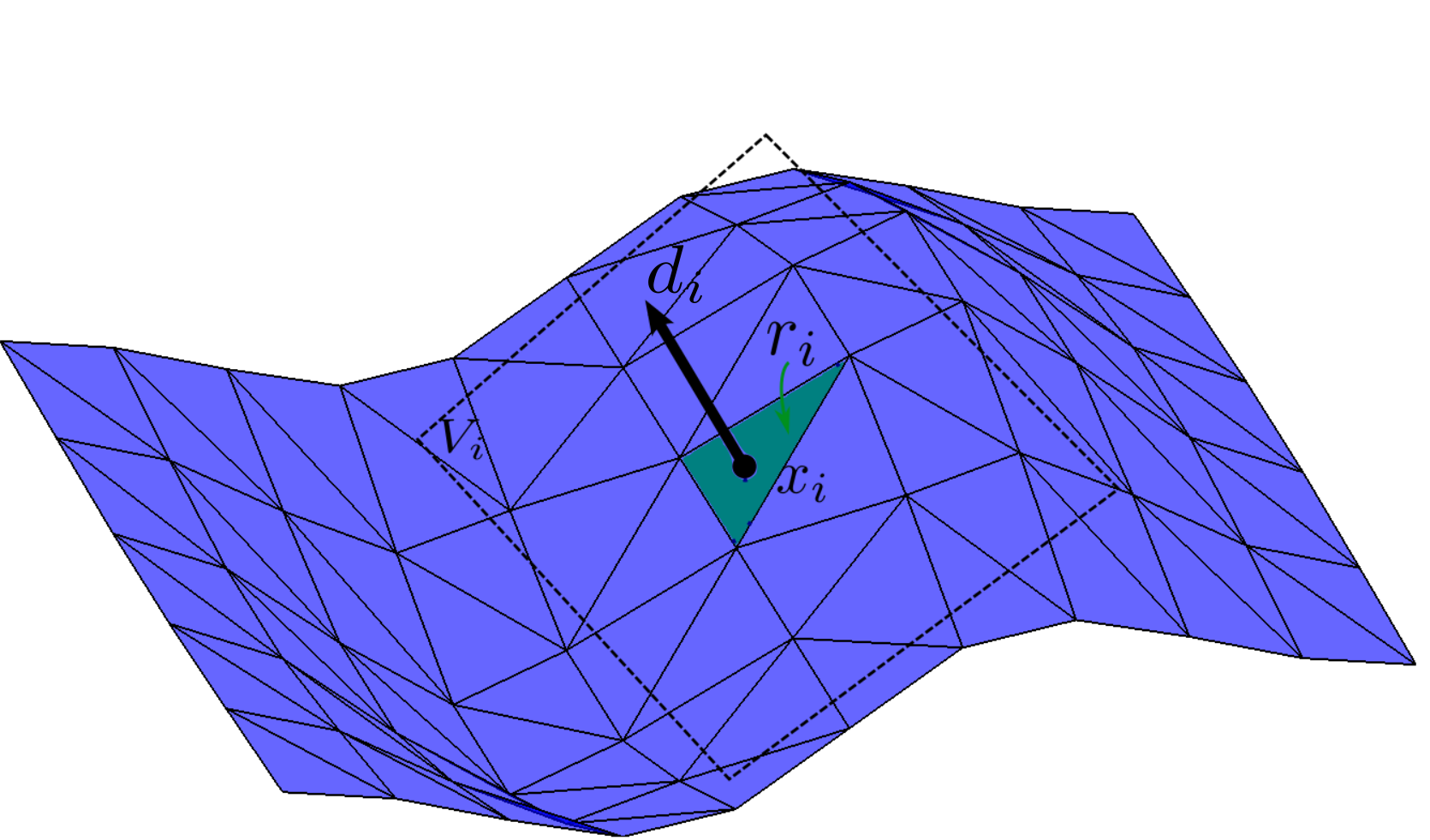} \\
    (a) & & (b) \\
    \includegraphics[width=4cm,height=4cm]{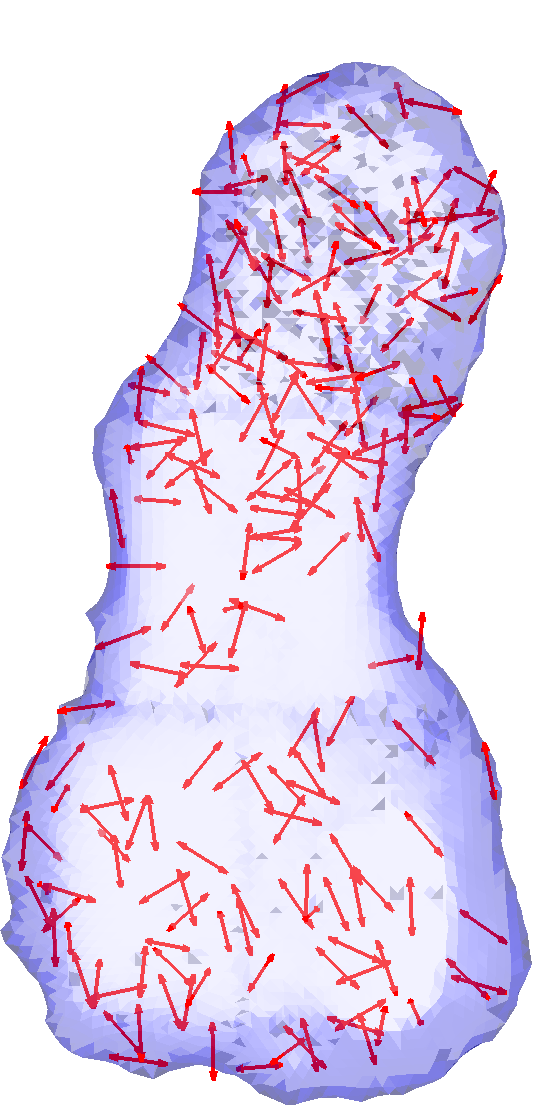} & & \includegraphics[width=4.5cm]{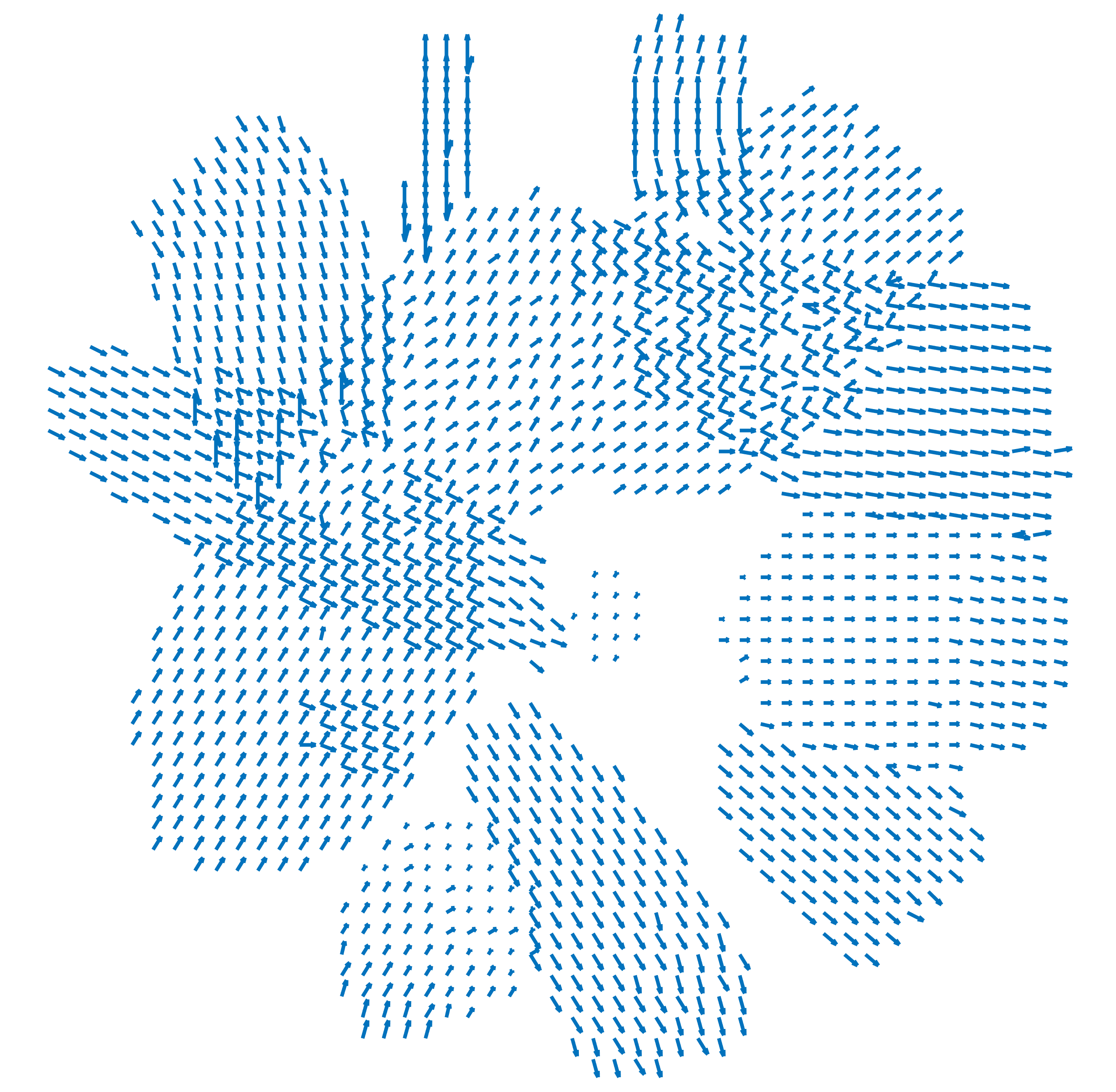} \\
    (c) & & (d)
    \end{tabular}
    \caption{Some examples of data representable by discrete varifolds: (a) Piecewise linear curve. (b) Triangulated surface. (c) A set of cells' mitosis directions measured inside a mouse embryonic heart membrane (c.f \cite{Ragni2017}). (d) Peak diffusion directions extracted from a slice of High Angular Resolution Diffusion Imaging phantom data, note the presence of multiple directions at certain locations corresponding to fiber crossing.}
    \label{fig:discr_shape_var}
\end{figure}

The relationship between shapes and varifolds relies on the fact that discrete shapes, namely curve or surface meshes, can be naturally approximated by varifolds of the previous form. As explained with more details in the aforementioned references, this is done by associating to any cell of the discrete mesh (i.e a segment for curves or a triangular face for surfaces) the weighted Dirac $r_i \delta_{(x_i,d_i)}$ as illustrated in Figure \ref{fig:discr_shape_var}. In that expression, $x_i$ is the coordinates of the center of the cell, $r_i$ its total length or area and $d_i$ the direction of the tangent space represented by the unit tangent or normal orientation vector $d_i$. It results in a mapping $S \mapsto \mu_{S}$ that associates to any discrete shape $S$ the discrete varifold $\mu_S = \sum_{i=1}^{F} r_i \delta_{(x_i,d_i)} \in \mathring{\mathcal{D}}$ obtained as the sum over all faces $i=1,\ldots,F$ of the corresponding Diracs. 

The main interest of such a representation is that it gives a convenient setting for the definition of shape similarities that are easy to compute without the need for pointwise correspondences between points. Assuming, which is quite natural in our context, that $W$ is a Hilbert space and that all Diracs $\delta_{(x,d)}$ for $(x,d) \in \R^n \times \S^{n-1}$ belong to the dual, $W$ must be then chosen as a Reproducing Kernel Hilbert Space (RKHS) associated to a smooth positive definite kernel on $\R^n \times \S^{n-1}$. In particular, we will follow the construction proposed in \cite{Charon2017} and consider separable kernels of the form $k(x,d,x',d')=\rho(|x-x'|^2) \gamma(\langle d , d' \rangle)$ where $\rho$ and $\gamma$ define positive definite kernel functions respectively on the positions between particles and the angles between their orientation vectors. The reproducing kernel metric on $W$ then gives a dual metric on varifolds that explicitly writes, for $\mu=\sum_{i=1}^{P} r_i \delta_{(x_i,d_i)}$:
\begin{equation}
\label{eq:metric_var}
 \|\mu\|_{W^*}^2 = \sum_{i,j} r_i r_j \rho(|x_i-x_j|^2) \gamma(\langle d_i, d_j \rangle)
\end{equation}

Such metrics on $W^*$ are determined by the choice of the positive definite functions $\rho$ and $\gamma$ and provide a global measure of proximity between two discrete varifolds. One important advantage for applications to e.g registration is that the computation of a distance $\|\mu-\mu'\|_{W^*}^2$ between two distributions does not require finding correspondences between their masses but instead reduces numerically to a quadratic number of kernel evaluations. The gradients of the metric with respect to the $x_i$'s and $d_i$'s is also very easy to obtain by direct differentiation of \eqref{eq:metric_var}. Finally, we note that the expression in \eqref{eq:metric_var} is also invariant to the action of the group of rigid motion. Namely for any rotation matrix $R$, translation vector $h$ and the group action $(R,h)\cdot \mu \doteq \sum_{i=1}^{P} r_i \delta_{(Rx_i+h,Rd_i)}$, one has $\|(R,h)\cdot \mu \|_{W^*} = \|\mu\|_{W^*}$. 

In all generality however, \eqref{eq:metric_var} may only yield a pseudo-metric on the set of discrete varifolds $\mathcal{D}$ since the inclusion mapping $\mathcal{D} \rightarrow W^*$ is not necessarily injective. A necessary and sufficient condition is: 
\begin{prop}
\label{prop:var_metric1}
 The metric $\|\cdot\|_{W^*}$ on $W^*$ induces a metric on $\mathcal{D}$ if and only if $k$ is a strictly positive definite kernel on $\R^n \times \S^{n-1}$. 
\end{prop}
The proof follows immediately from the definition of strictly positive definite kernel. This condition holds in particular if both kernels defined by $\rho$ and $\gamma$ are strictly positive definite. In the case of $\mathring{\mathcal{D}}$, one can provide different sufficient conditions which are often more convenient to satisfy in practice. These involve a density property on kernels called $C_0$-universality, cf \cite{Carmeli2010}. A kernel on $\R^n$ is said to be $C_0$-universal if the associated RKHS is dense in $C_0(\R^n,\R)$. Then one has the following 
\begin{prop}
\label{prop:var_metric2}
 If the kernel defined by $\rho$ is $C_0$-universal, $\gamma(1)>0$ and $\gamma(u) < \gamma(1)$ for all $u \in [-1,1)$, then $\|\cdot\|_{W^*}$ induces a metric on $\mathring{\mathcal{D}}$. 
\end{prop}
\begin{proof}
Let $W_{pos}$ and $W_{or}$ be the RKHS associated to $\rho$ and $\gamma$. By contradiction, suppose that $\mu, \mu' \in \mathring{\mathcal{D}}$ with $\|\mu-\mu'\|_{W^*}=0$ and $\mu \neq \mu'$ in $\mathring{\mathcal{D}}$. We can write $\mu,\mu'$ in the following form:
\begin{align*}
    \mu = \sum_{i=1}^N r_i \delta_{(z_i,d_i)}, \ \ \mu' = \sum_{i=1}^N r_i' \delta_{(z_i,d_i')},
\end{align*}
where $\{z_i\}$, with $z_i$ all distinct, is the reunion of point positions from both distributions and $\max\limits_{1 \leq i \leq N} \{r_i,r_i' \}>0$, $\min\limits_{1 \leq i \leq N} \{r_i,r_i' \} \geq 0$. Since $\mu$ and $\mu'$ are distinct in $\mathring{\mathcal{D}}$, there is some $i_0$ such that $(d_{i_0},r_{i_0}) \neq (d_{i_0}',r_{i_0}')$. Without loss of generality, we may assume $r_{i_0} \geq r_{i_0}'$. Let $g(\cdot) = \gamma(\langle d_{i_0}, \cdot \rangle) \in W_{or}$ and choose $f \in C_0(\mathbb{R}^n,\mathbb{R})$ satisfying $f(z_{i_0}) =1$ and $f(z_i)=0$ for all $i \geq 2$. Since the kernel defined by $\rho$ is $C_0$-universal, there exists $\{f_n\} \subset W_{pos}$ such that $f_n \rightarrow f$ uniformly. As $f_n \otimes g \in W$, we have that
\begin{equation*}
    0 = (\mu-\mu'|f_n \otimes g) = \sum_{i=1}^N f_n(z_i) (r_i g(d_i) -r_i' g(d_i'))
\end{equation*}
Taking the limit $n \rightarrow +\infty$, this gives:
\begin{equation}
\label{eq:proof_prop2}
    0 = f(z_{i_0}) (r_{i_0} g(d_{i_0}) -r_{i_0}' g(d_{i_0}')) = \underbrace{r_{i_0} \gamma(1) - r_{i_0}' \gamma(\langle d_{i_0}, d_{i_0}'\rangle)}_{A} .
\end{equation}
Since $(d_{i_0},r_{i_0}) \neq (d_{i_0}',r_{i_0}'), \ r_{i_0} \geq r_{i_0}' \textrm{ and } r_{i_0}>0$, we have either $d_{i_0}\neq d_{i_0}'$ and then $A \geq r_{i_0}(\gamma(1) -\gamma(\langle d_{i_0}, d_{i_0}'\rangle))>0$ or $d_{i_0} = d_{i_0}'$ and $r_{i_0}> r_{i_0}'$ in which case $A = (r_{i_0} - r_{i_0}') \gamma(1)>0$. In either case the right hand side of \eqref{eq:proof_prop2} is positive which is a contradiction.
\end{proof}
Note that the $C_0$-universality assumption still implies that the kernel defined by $\rho$ is strictly positive definite. However, the assumptions on $\gamma$ are typically less restrictive than in Proposition \ref{prop:var_metric1}. 

A last subclass of varifold metrics that shall be of interest in this paper is the case of orientation-invariant kernels which amounts in choosing an even function $\gamma$ in the kernel definition. This, indeed, leads to a space $W^*$ and metric $\|\cdot\|_{W^*}$ for which Diracs $\delta_{(x,d)}$ and $\delta_{(x,-d)}$ are equal in $W^*$ for any $(x,d) \in \R^n \times \S^{n-1}$. In other words, elements of $\mathcal{D}$ can be equivalently viewed as \textit{unoriented} varifolds, i.e distributions on the product of $\R^n$ and the projective space of $\R^n$, similarly to the framework of \cite{Charon2}. In that particular situation, one obtains an induced distance under the conditions stated in the following proposition which proof is a straightforward adaptation of the one of Proposition \ref{prop:var_metric2}. 
\begin{prop}
\label{prop:var_metric3}
 If the kernel defined by $\rho$ is $C_0$-universal, $\gamma$ is an even function with $\gamma(1)>0$ and $\gamma(u) < \gamma(1)$ for all $u \in (-1,1)$, then $\|\cdot\|_{W^*}$ induces a metric on the space $\mathring{\mathcal{D}}$ modulo the orientation. 
\end{prop}

In Section \ref{sec:results} below, we will discuss more thoroughly and illustrate the effects of those kernel properties on the solutions to registration problems for different cases of discrete distributions. 


\section{Optimal diffeomorphic mapping of varifolds}
\label{sec:optimal_control}
It is essential to point out that the notion of varifold presented above contains but is also more general than curves and surfaces as it allows to model more complex geometric structures like objects carrying multiple orientation vectors at a given position. In contrast with most previous works on diffeomorphic registration that only involve varifolds as an intermediary representation to compute fidelity terms between shapes, the purpose of this paper to derive a deformation model and registration framework on the space $\mathcal{D}$ itself.  

\subsection{Group action}
\label{ssec:group_action}
A first key element is to express the way that deformations 'act' on discrete varifolds. Considering a smooth diffeomorphism $\phi \in \text{Diff}(\R^n)$, we first intend to express how $\phi$ should transport a Dirac $\delta_{(x,d)}$. There is however not a canonical way to define it as the nature of the underlying data affects the deformation model itself. An important distinction to be made is on the interpretation of direction vectors $d$, whether they correspond for instance to a unit tangent direction to a curve or a surface in which case $d$ is transported by the Jacobian of $\phi$ as $D_x \phi(d)/|D_x \phi(d)|$ or rather to a normal direction which instead requires a transport model involving the inverse of the transposed Jacobian i.e $(D_x \phi)^{-T}(d)/|(D_x \phi)^{-T}(d)|$ (see \cite{Younes} chap. 10 for more thorough discussion). To keep notations more compact, we will write $D\phi \cdot d$ for a given generic action of $D\phi$ on $\R^{n}$ on either tangent or normal vector and $\overline{D\phi \cdot d}$ for the corresponding normalized vector in $\S^{n-1}$. That being said, we will also consider two distinct models for the action:
\begin{itemize}
    \item[$\bullet$] $\phi_{*} \delta_{(x,d)} \doteq \delta_{(\phi(x),\overline{D\phi\cdot d})}$ (\textit{normalized action}): this corresponds to transporting the Dirac mass at the new position $\phi(x)$ and transforming the orientation vector as $\overline{D\phi \cdot d}$. 
    
    \item[$\bullet$] $\phi_{\#} \delta_{(x,d)} \doteq |D\phi\cdot d| \delta_{(\phi(x),\overline{D\phi\cdot d})}$ (\textit{pushforward action}): the position and orientation vector are transported as previously but with a reweighting factor equal to the norm of $D\phi\cdot d$.
\end{itemize}
It is then straightforward to extend both of these definitions by linearity to any discrete varifold in $\mathcal{D}$. In both cases, we obtain a group action of diffeomorphisms on the set of discrete varifolds. However, these actions are clearly not equivalent. The normalized action operates as a pure transport of mass and rotation of the direction vector whereas the pushforward model adds a weight change corresponding to the Jacobian of $\phi$ along the direction $d$. This is a necessary term in the situation where $\mu=\mu_{S}$ is representing a discrete oriented curve or surface. Indeed, one can check, up to discretization errors, that under the pushforward model, we have $\phi_{\#} \mu_{S} = \mu_{\phi(S)}$; in other words the action is compatible with the usual deformation of a shape. In the result section below, we will show examples of matching based on those different group action models. 

Although we will be focusing on special subgroups of diffeomorphisms in the next section, it will be insightful to study a little more closely the orbits of discrete varifolds under the normalized and pushforward actions of the full group $\text{Diff}(\R^n)$ (or similarly the equivalence classes $\mathcal{D}/\text{Diff}(\R^n)$). Let $\mu \in \mathcal{D}$ which we can write as $\mu = \sum_{i=1}^{N} \sum_{j=1}^{n_i} r_{i,j} \delta_{(x_{i},d_{i,j})}$ where the $x_i$ are here assumed to be distinct positions and for each $i=1,\ldots,N$, the $(d_{i,j})_{j=1,\ldots,n_i}$ are distinct in $\S^{n-1}$. While it is well-known that $\text{Diff}(\R^n)$ acts transitively on the set of point clouds of $N$ points in $\R^n$ (as $n\geq2$), this may no longer hold when one or several direction vectors are attached to each point position. 

In the case of the normalized action, we have $\phi_{*} \mu = \sum_{i=1}^{N} \sum_{j=1}^{n_i} r_{i,j} \delta_{(\phi(x_i),\overline{D\phi\cdot d_{i,j}})}$. We see that the orbit of $\mu$ is then given by: 
\begin{equation*}
    \text{Diff}_{*}\mu = \left \{\sum_{i=1}^{N} \sum_{j=1}^{n_i} r_{i,j} \delta_{(y_i,u_{i,j})} \ s.t \ y_i \neq y_j \ \text{for} \ i \neq j, \ \exists A_1,\ldots,A_{N} \in \text{GL}(\R^n), \ u_{i,j}= \frac{A_i d_{i,j}}{|A_i d_{i,j}|}  \right \} 
\end{equation*}
This is essentially the set of all discrete varifolds with any set of $N$ distinct positions and for each $i$, a set of $n_i$ directions obtained by a linear transformation of the $\{d_{i,j}\}_{j=1,\ldots,n_i}$ with  weights $r_{i,j}$ unchanged. In particular, this imposes some constraints on the set of 'attainable' direction vectors: clearly, if the number of direction vectors at a given position exceeds the dimension i.e $n_i \geq n$, this system of vectors cannot be mapped in general to any other system of $n_i$ vectors on the sphere by a single linear map. If we assume that the system of vectors at each position $x_i$ forms a frame, i.e that for all $i$, $n_i \leq n$ and the direction vectors $d_{i,j}$ for $j=1,\ldots,n_i$ are independent, then we see that the orbit of $\mu$ is given by the set of all discrete varifolds of the form $\sum_{i=1}^{N} \sum_{j=1}^{n_i} r_{i,j} \delta_{(y_i,u_{i,j})}$ with distinct $y_i$'s and $(u_{i,j})$ in $\S^{n-1}$ such that the $(u_{i,j})_{j=1,\ldots,n_i}$ are independent for all $i$. In the special case of $n_i=1$ for all $i$, that is $\mu \in \mathring{\mathcal{D}}$, the orbits are then entirely determined by the set of weights $r_i$ which gives the identification of $\mathring{\mathcal{D}}/\text{Diff}(\R^n)$ with ordered finite sets of positive numbers.  

With the pushforward action, we have $\phi_{\#} \mu = \sum_{i=1}^{N} \sum_{j=1}^{n_i} |D\phi\cdot d_{i,j}| r_{i,j} \delta_{(\phi(x_i),\overline{D\phi\cdot d_{i,j}})}$ and the orbit writes:
\begin{align*}
    \text{Diff}_{\#}\mu = \Big\{ \nu \in \mathcal{D} \ \ &s.t \ \ \exists (y_i) \in (\R^n)^N, \ y_i \neq y_j \ \text{for} \ i \neq j, \exists A_1,\ldots,A_{N} \in \text{GL}(\R^n),  \\ 
     &\nu=\sum_{i=1}^{N} \sum_{j=1}^{n_i} |A_i d_{i,j}| r_{i,j} \delta_{(y_i,u_{i,j})} \ \ \text{with} \ u_{i,j}= \frac{A_i d_{i,j}}{|A_i d_{i,j}|}  \Big\} 
\end{align*}
In the general situation, there is again no simple characterization of the orbit. With the additional assumptions that $n_i \leq n$ and the $(d_{i,j})_{j=1,\ldots,n_i}$ are independent vectors for each $i$, the orbit of $\mu$ is the set of all discrete varifolds of the form $\sum_{i=1}^{N} \sum_{j=1}^{n_i} s_{i,j} \delta_{(y_i,u_{i,j})}$ with any choice of distinct points $y_i$, direction vectors $(d_{i,j})$ in $\S^{n-1}$ such that the $(d_{i,j})_{j}$ are independent and weights $s_{i,j} > 0$. In particular, the action of $\text{Diff}(\R^n)$ in the pushforward model is transitive on all subsets of $\mathring{\mathcal{D}}$ with fixed $N$, which implies that the equivalence classes of $\mathring{\mathcal{D}}/\text{Diff}(\R^n)$ in that case are only determined by the number of Diracs in the discrete varifold, as we would expect.  

The previous discussion thus shows that for both models and unlike the more standard cases of landmarks or discrete vector fields, the action of diffeomorphisms on discrete varifolds is in general not transitive. It is therefore necessary to formulate the registration problems in their inexact form by introducing fidelity terms like the kernel metrics introduced in Section \ref{sec:discrete_varifolds}.     

\subsection{Optimal control problem}
\label{ssec:optimal_control}
With the definitions and notations of the previous sections, we can now introduce the mathematical formulation of the diffeomorphic registration of discrete varifolds. As mentioned in the introduction, we will rely on the LDDMM model for generating diffeomorphisms although other transformation spaces and models could be taken as well. In short, we consider a space of time dependent velocity fields $v \in L^2([0,1],V)$ such that for all $t\in [0,1]$, $v_t$ belongs to a certain RKHS $V$ of vector fields on $\R^n$. We will write $K: \ \R^n \times \R^n \rightarrow \R^n$ the vector-valued reproducing kernel of $V$. From $v$, one obtains the flow mapping $\phi^v_t$ at each time $t$ as the integral of the differential equation $\partial_t \phi^v_t = v_t \circ \phi^v_t$  with $\phi_0 = \text{Id}$. We then define our deformation group as the set of all flow maps $\phi_1^v$ for all velocity fields $v \in L^2([0,1],V)$. With the adequate assumptions on the kernel of $V$, this is a subgroup of the group of diffeomorphisms of $\R^n$ and it is naturally equipped with the metric given by $\int_{0}^{1} \|v_t\|_V^2 dt$, cf \cite{Younes} for a detailed exposition of the LDDMM framework. 

Now, let's consider two discrete varifolds $\mu_0 = \sum_{i=1}^{P} r_{i,0} \delta_{(x_{i,0},d_{i,0})}$ (template) and $\tilde{\mu} = \sum_{j=1}^{Q} \tilde{r}_{j} \delta_{(\tilde{x}_{j},\tilde{d}_{j})}$ (target). We formulate the inexact matching problem between $\mu_0$ and $\tilde{\mu}$ as follows:
\begin{equation}
\label{eq:matching_var}
\text{argmin}_{v \in L^2([0,1],V)} \left\{ E(v) = \int_{0}^{1} \|v_t\|_V^2 dt + \lambda \|\mu(1) - \tilde{\mu} \|_{W^*}^2 \right\}
\end{equation}
subject to either $\mu(t) \doteq (\phi_t^v)_{*} \mu_0$ in the normalized action scenario or $\mu(t) \doteq (\phi_t^v)_{\#} \mu_0$ for the pushforward model, and $\lambda$ being a weight parameter between the regularization and fidelity terms in the energy. This is easily interpreted as an optimal control problem in which the state variable is the transported varifold $\mu(t)$, the control is the velocity field $v$ and the cost functional is the sum of the standard LDDMM regularization term on the deformation and a discrepancy term between $\mu(1)$ and the target given by a varifold kernel metric as in \eqref{eq:metric_var}. Those optimal control problems are well-posed in the following sense:
\begin{prop}
If $V$ is continuously embedded in the space $C^2_0(\R^n,\R^n)$, or equivalently if $K$ is of class $C^2$ with all derivatives up to order 2 vanishing at infinity, then there exists a global minimum to the problem \eqref{eq:matching_var}.  
\end{prop}
\begin{proof}
The result follows from an argument similar to that of the existence of minimizers in usual LDDMM registration problems. If $(v^{n})$ is a minimizing sequence in $L^2([0,1],V)$ then thanks the first term of $E$, we may assume that $(v^n)$ is bounded in $L^2([0,1],V)$ and therefore that, up to extracting a subsequence, $v^n \rightharpoonup v^*$ weakly in $L^2([0,1],V)$. It then follows from the results of \cite{Younes} (Chapter 8.2) that the sequence of diffeomorphisms $(\phi_1^{v^n})$ and their first-order differentials $(d\phi_1^{v^n})$ converge uniformly on every compact respectively to $\phi_{1}^{v^*}$ and $d\phi_{1}^{v^*}$. In particular, for all $i=1,\ldots,P$, $\phi_{1}^{v^n}(x_i) \rightarrow \phi_{1}^{v^*}(x_i)$ and $d\phi_{1}^{v^n}(x_i) \rightarrow d\phi_{1}^{v^*}(x_i)$. Then, from the expressions of the group actions and the metric \eqref{eq:metric_var}, we obtain that either $\|(\phi_1^{v^n})_{*} \mu_0 - \tilde{\mu} \|_{W^*}^2 \xrightarrow[n\rightarrow \infty]{} \|(\phi_1^{v^*})_{*} \mu_0 - \tilde{\mu} \|_{W^*}^2$ or $\|(\phi_1^{v^n})_{\#} \mu_0 - \tilde{\mu} \|_{W^*}^2 \xrightarrow[n\rightarrow \infty]{} \|(\phi_1^{v^*})_{\#} \mu_0 - \tilde{\mu} \|_{W^*}^2$. Finally, using the weak lower semicontinuity of the norm in $L^2([0,1],V)$, it gives in both cases:
\begin{equation*}
    E(v^*) \leq \lim \inf_{n \rightarrow \infty} E(v^n)
\end{equation*}
and consequently $v^*$ is a global minimizer of $E$. 
\end{proof}

\subsection{Hamiltonian dynamics}
\label{ssec:hamiltonian_dynamics}
By fixing the final time condition $\mu(1)$ and minimizing $\int_{0}^{1} \|v_t\|_V^2 dt$ with those boundary constraints, the resulting path $t \mapsto \mu(t)$ corresponds to a \textit{geodesic} in $\mathcal{D}$ for the metric induced by the metric on the deformation group. We can further characterize those geodesics as solutions of a Hamiltonian system. For that purpose, we follow the general setting developed in \cite{arguillere14:_shape} for similar optimal control problems. 

In our situation, we can describe the state $\mu(t)$ as a set of $P$ particles each given by the triplet $(x_i(t),d_i(t),r_i(t)) \in \R^n \times \S^{n-1} \times \R_{+}^{*}$ representing its position, orientation vector and weight. From \ref{ssec:group_action}, we have that $x_i(t) = \phi_t^v(x_{i,0})$, $d_i(t) = \overline{D\phi^v_t\cdot d_{i,0}}$ and $r_i(t) = r_{i,0}$ for the normalized action and $r_i(t)=|D\phi^v_t\cdot d_{i,0}| r_{i,0}$ in the pushforward case. Differentiating with respect to $t$, the state evolution may be alternatively described by the set of ODEs 
$$
\left \{
\begin{array}[h]{l}
\dot{x}_i (t)= v_t(x_i(t)) \\
\dot{d}_i (t) = P_{d_i(t)^{\bot}}(D v_t\cdot d_{i}(t)) \\
\dot{r}_i (t) = \left\{ 
\begin{array}[h]{l}
0 \ \ \text{(normalized)} \\
\langle d_i(t), D v_t\cdot d_{i}(t) \rangle r_i(t) \ \ \text{(pushforward)}
\end{array}
\right.
\end{array}
 \right.
$$
 
\noindent where $P_{d_i(t)^\bot}$ denotes the orthogonal projection on the subspace orthogonal to $d_i(t)$, $D v_t \cdot d_i(t)$ corresponds to the infinitesimal variation of the action of $D\phi$ on vectors of $\R^n$ introduced in \ref{ssec:group_action}: it is given specifically by $D v_t \cdot d_i(t) = D_{x_i(t)} v_t (d_i(t))$ in the tangent case and $D v_t \cdot d_i(t) = -(D_{x_i(t)} v_t)^{T}(d_i(t))$ in the normal case. Note that other choices of transformation of the weights could be treated quite similarly by modifying accordingly the last equation in the previous system. In what follows, we detail the derivations of the optimality equations in the case of tangent direction vectors for both normalized and pushforward group action models, the situation of normal vectors being easily tackled in similar fashion.    

\subsubsection{Normalized action}
In the case of normalized action, $\{r_i(t)\}_{i=1}^P$ are time independent as previous discussed. So we can choose the state variable of the optimal control problem to be $q:= \{(x_i,d_i), \ i =1,\cdots,P\} \in \mathbb{R}^{2dP}$ with the infinitesimal action
\begin{align*}
\xi_q v = \left\{\left( v(x_i),P_{d_i^{\perp}}(D_{x_i}v(d_i))\right),\ i=1,\cdots,P\right\}
\end{align*} 
and introduce the Hamiltonian

\begin{align*}
H(p,q,v) &= (p|\xi_qv) -\frac{1}{2}\|v\|^2_V \\
         &=\sum_{i=1}^N \langle p_i^{(1)}, v(x_i) \rangle + \langle P_{d_i^{\perp}}(p_i^{(2)}), D_{x_i}v(d_i) \rangle -\frac{1}{2}\|v\|^2_V,
\end{align*}
where 
\begin{align*}
p= \left\{\left(p_i^{(1)},p_i^{(2)}\right), \ i=1,\cdots,P \right\} \in \mathbb{R}^{2dP}
\end{align*}
is the adjoint variable of state $q$. We call $p_i^{(1)}$ the \textit{spatial momentum} and $p_i^{(2)}$ the \textit{directional momentum}. From Pontryagin's maximum principle, the Hamiltonian dynamics is given by the forward system of equations
\begin{align}
\left\{\begin{array}{ll}
\dot{x}_i(t) &= v_t(x_i(t)) \\
\dot{d}_i(t) &= P_{d_i(t)^{\perp}}(D_{x_i(t)}v_t(d_i(t))) \\
\dot{p}_i^{(1)}(t) &= -(D_{x_i(t)}v_t)^T p_i^{(1)}(t) \\
             &-(D_{x_i(t)}^{(2)}v_t(\cdot,d_i(t)))^T\left( P_{d_i(t)^{\perp}}(p_i^{(2)}(t))\right) \\
\dot{p}_i^{(2)}(t) &= -(D_{x_i(t)} v_t)^T P_{d_i(t)^{\perp}}(p_i^{(2)}(t)) \\
                &+ \langle d_i(t),p_i^{(2)}(t) \rangle D_{x_i(t)}v_t(d_i(t)) \\
                &+ \langle d_i(t),D_{x_i(t)}v_t(d_i(t)) \rangle p_i^{(2)}(t)
\end{array}\right.
\label{eq:forward_normalized}
\end{align}
and optimal vector fields $v$ satisfy
\begin{align*}
\langle v_t, h \rangle_V &= \left( p(t),\xi_{q(t)}h \right) \\
&= \sum_{i=1}^P \langle p_i^{(1)},h(x_i)\rangle + \langle P_{d_i(t)^{\perp}}(p_i^{(2)}(t)),D_{x_i(t)}h(d_i(t))\rangle, 
\end{align*}
for any $h \in V$ and $t \in [0,1]$. The reproducing property and reproducing property for the derivatives in a vector RKHS give \cite{Sommer2013} that $\forall x \in \R^n, \ z \in \mathbb{R}^n, v \in V$ and multi-index $\alpha$,
\begin{align*}
&\langle K(x,\cdot)z,v \rangle_V = (z \otimes \delta_x |v) \\
&\langle D_1^{\alpha} K(x,\cdot),v \rangle_V = z^T D^{\alpha} v(x).
\end{align*}
With the above properties, we obtain the following expression of $v$
\begin{align}
\label{eq:v_normalized}
v_t(\cdot) &= \sum_{k=1}^P K(x_k(t),\cdot)p_k^{(1)}(t) \nonumber \\
&+ D_1 K(x_k(t),\cdot)\left(d_k(t),P_{d_k(t)^{\perp}}(p_k^{(2)}(t))\right).
\end{align}
where we use the shortcut notation $D_1 K(x,\cdot)(u_1,u_2)$ for the vector $D_1 (K(x,\cdot) u_2)(u_1)$. In Figure \ref{fig:normalized}, we show an example of geodesic and resulting deformation for a single Dirac varifold, which is obtained as the solution of \eqref{eq:forward_normalized} with the initial momenta shown in the figure. It illustrates the combined effects of the spatial momentum which displaces the position of the Dirac and of the directional momentum that generates a local rotation of the direction vector.

\begin{figure}
    \centering
    \begin{tabular}{ccc}
    \includegraphics[trim = 15mm 15mm 15mm 15mm ,clip,width=4cm]{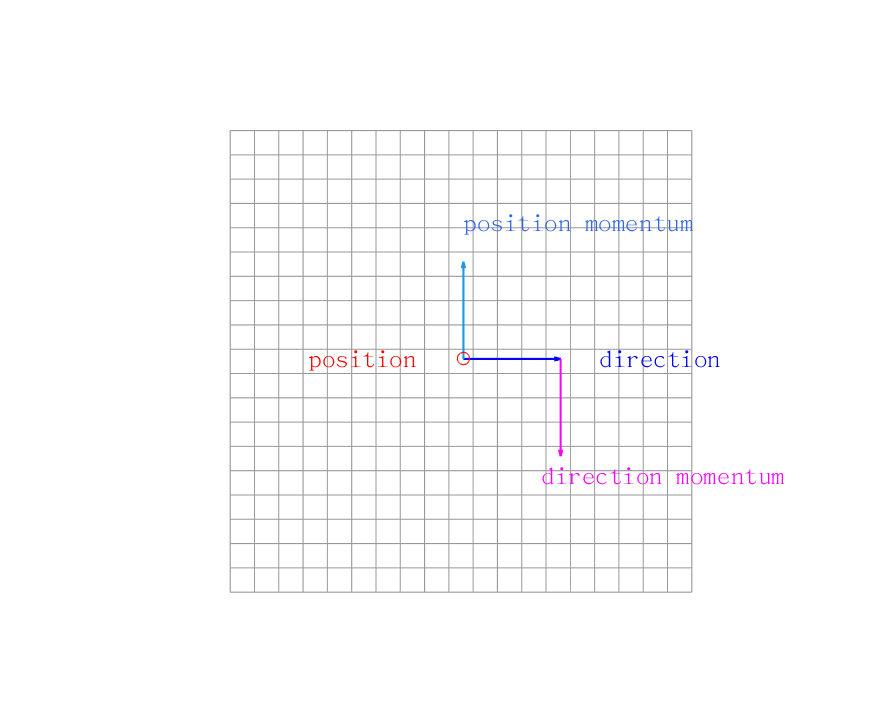} 
    &\includegraphics[trim = 15mm 15mm 15mm 15mm ,clip,width=4cm]{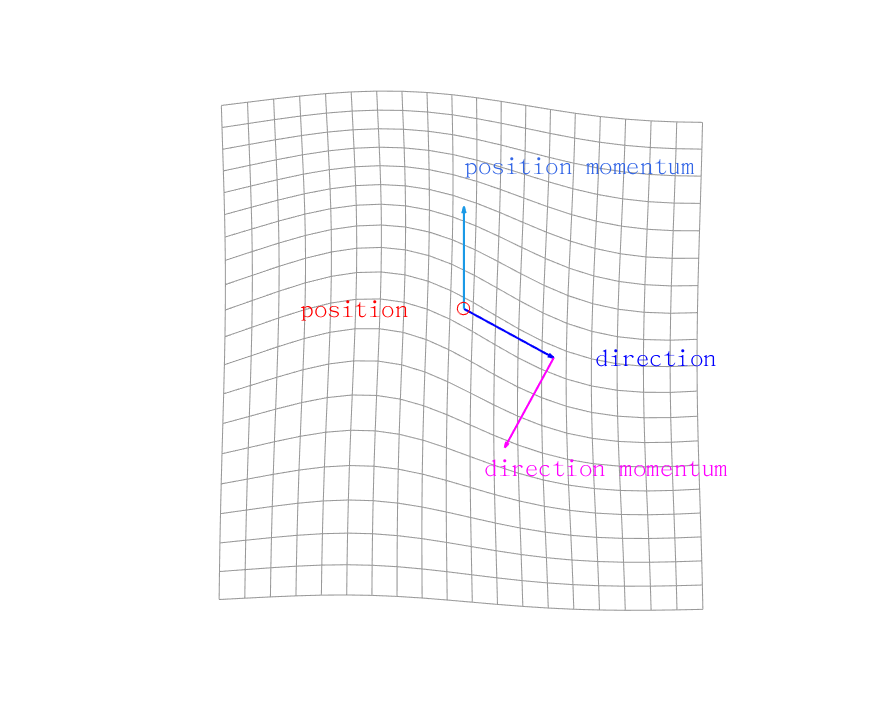} 
    &\includegraphics[trim = 15mm 15mm 15mm 15mm ,clip,width=4cm]{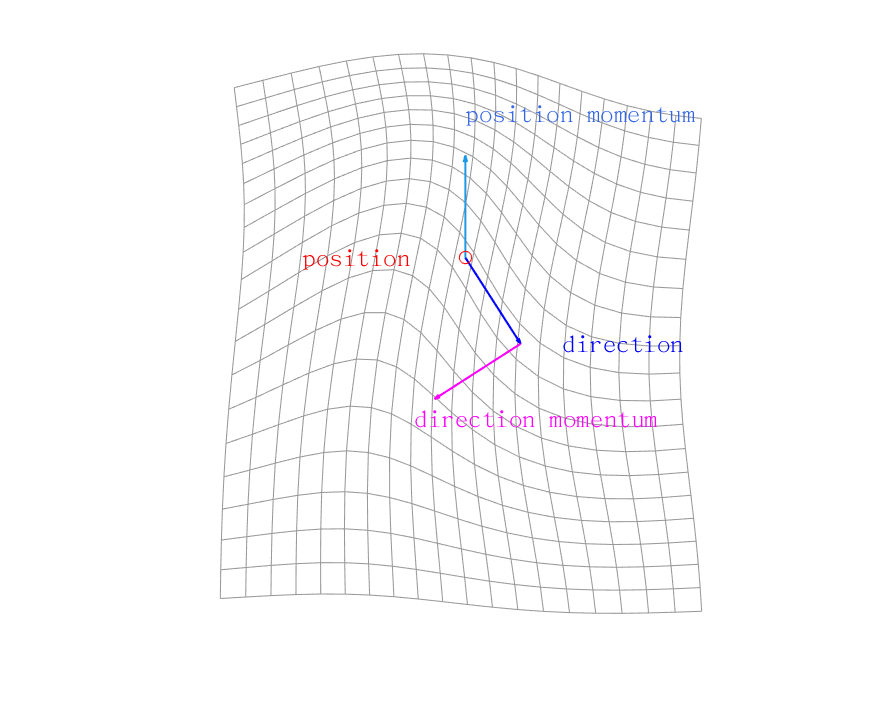}\\
    $t=0$ & $t=1/2$& $t=1$  \\
    \end{tabular}
    \caption{Example of geodesic for a single Dirac in the normalized action case.}
    \label{fig:normalized}
\end{figure}

\subsubsection{Pushforward action}
As in the previous section, we set the state variable $q := \{(x_i,d_i,r_i), \  i=1,\cdots,P\} \in (\R^n \times \S^{n-1} \times \R_{+}^{*} )^P$, the infinitesimal action
\begin{align*}
\xi_q v = \left\{ \left(v(x_i),P_{d_i^{\perp}} (d_{x_i}v(d_i)), r_i \langle d_i, d_{x_i}v(d_i) \rangle \right), \ i=1,\cdots,P \right\}    
\end{align*}
and the Hamiltonian
\begin{align}\label{fullHam}
H(p,q,v) = \sum_{i=1}^P \left\langle p_i^{(1)}, v(x_i) \right\rangle + \left\langle P_{d_i^{\perp}}(p_i^{(2)}), d_{x_i}v(d_i) \right\rangle + p_i^{(3)} r_i \left\langle  d_i, d_{x_i}v(d_i) \right\rangle - \frac{1}{2} \|v\|_V^2,
\end{align}
where $p = \left\{ \left( p_i^{(1)}, p_i^{(2)}, p_i^{(3)} \right), \ i=1,\cdots,P \right\} \in \R^{(2d+1)P}$. Applying again Pontryagin's maximum principle, we obtain the forward system
\begin{align} \label{fullforward}
\left\{
\begin{array}{ll}
\dot x_i &=  v_t(x_i) \\
\dot d_i &= P_{d_i^{\perp}} (d_{x_i} v(d_i)) \\
\dot r_i &= r_i \left\langle d_i, d_{x_i} v_t(d_i) \right\rangle\\
\dot p_i^{(1)} &= -(d_{x_i(t)}v_t)^T p_i^{(1)}-(d_{x_i(t)}^{(2)}v(\cdot,d_i))^T (P_{d_i^{\perp}}(p_i^{(2)}) ) - p_i^{(3)} r_i d_{x_i}v(\cdot,d_i)^T d_i\\
\dot p_i^{(2)} &= -d_{x_i(t)}v_t^T(p_i^{(2)}) + \left( \left\langle  d_i,p_i^{(2)}\right\rangle - r_i p_i^{(3)} \right) \left[ d_{x_i}v_t + d_{x_i}v^T_t \right] (d_i) \\
          &+  \left\langle d_i, d_{x_i}v_t(d_i) \right\rangle p_i^{(2)} \\
\dot p_i^{(3)} &= - p_i^{(3)} \left\langle d_i, d_{x_i}v_t(d_i) \right\rangle
\end{array} \right. 
\end{align}
with optimal vector field of the form
\begin{align} \label{fullvf}
v_t(x) = \sum_{k=1}^P K(x_k,x) p_k^{(1)} + D_1K(x_k,x)\left(d_k,P_{d_k^{\perp}}(p_k^{(2)})+p_k^{(3)}r_k d_k \right) .
\end{align}
From the forward equations \eqref{fullforward}, we see that $\frac{d}{dt} \langle d_i(t), d_i(t) \rangle =0$ and $\frac{d}{dt} r_i(t) p_i^{(3)}(t) =0$, hence $\|d_i(t)\|$ and $ r_i(t) p_i^{(3)}(t)$ are constant along geodesic paths. Similarly to the normalized action case, we can use use those conservation properties to reduce the number of state and dual variables as follows. 

Let the new state variable be $q = \{(x_i,u_i), \ i=1,\cdots,P\}$ and the Hamiltonian
\begin{align} \label{reduced_ham}
H(p,q,v) = \sum_{i=1}^P  \langle p_i^{(1)}, v(x_i) \rangle +  \langle p_i^{(2)}, d_{x_i}v(u_i)\rangle -  \frac{1}{2} \| v\|_V^2
\end{align}
The forward equations and optimal vector field $v$ derived from this Hamiltonian are
\begin{align} \label{reduced_foward}
\left\{
\begin{array}{l}
\dot x_i(t)  = v_t(x_i(t)) \\
\dot d_i(t) = d_{x_i(t)}v_t(u_i(t)) \\
\dot p_i^{(1)}(t) =  -(d_{x_i(t)}v_t)^T p_i^{(1)}-(d_{x_i(t)}^{(2)}v(\cdot,u_i))^T p_i^{(2)} \\
\dot p_i^{(2)}(t) = -(d_{x_i(t)}v_t)^T p_i^{(2)}(t)
\end{array} \right. 
\end{align}
and 
\begin{align} \label{reduced_vf}
v_t(x) = \sum_{k=1}^P K(x_k,x)p_k^{(1)}(t) + D_1 K(x_k,x)(u_k(t),p_k^{(2)}(t)).    
\end{align}
Then this new system is rigorously equivalent to the original one in the following sense:
\begin{prop} \label{reduced_prop}
Any solution of \eqref{reduced_foward} + \eqref{reduced_vf} is such that $\left( x_i(t),\overline{u_i(t)},|u_i(t)| \right)$ is a solution of \eqref{fullforward} + \eqref{fullvf}. Conversely, any solution $\left(x_i(t),d_i(t),r_i(t) \right)$ of \eqref{fullforward} + \eqref{fullvf} with initial conditions satisfying $\left\langle p_i^{(2)}(0), u_i(0) \right\rangle = r_i(0) p_i^{(3)}(0)$ gives the solution $\left(x_i(t),r_i(t)d_i(t)\right)$ to \eqref{reduced_foward} + \eqref{reduced_vf}.
\end{prop}

\begin{figure}
    \centering
    \begin{tabular}{ccc}
    \includegraphics[trim = 12mm 12mm 12mm 12mm ,clip,width=4cm]{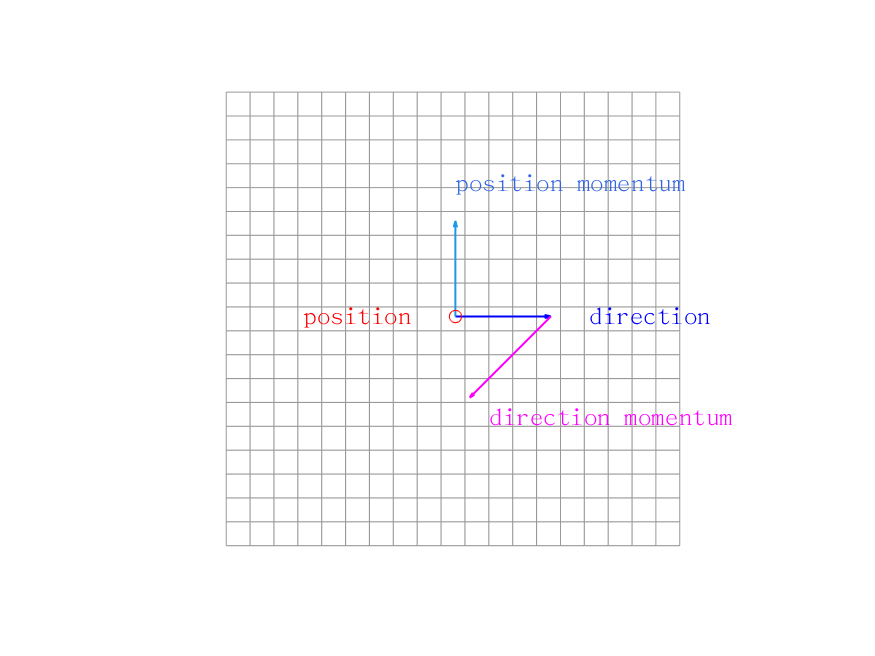} 
    &\includegraphics[trim = 12mm 12mm 12mm 12mm ,clip,width=4cm]{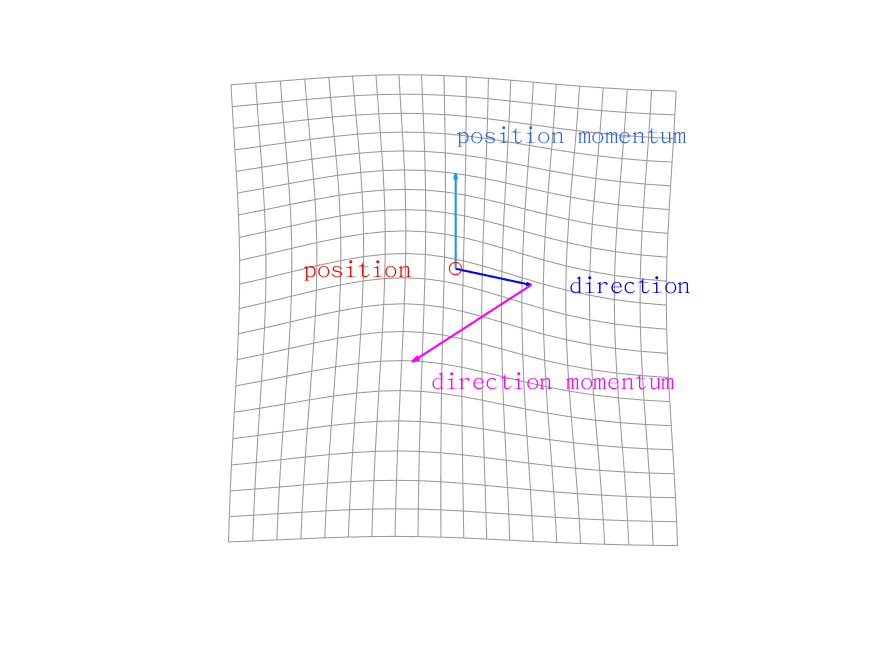} 
    &\includegraphics[trim = 12mm 12mm 12mm 12mm ,clip,width=4cm]{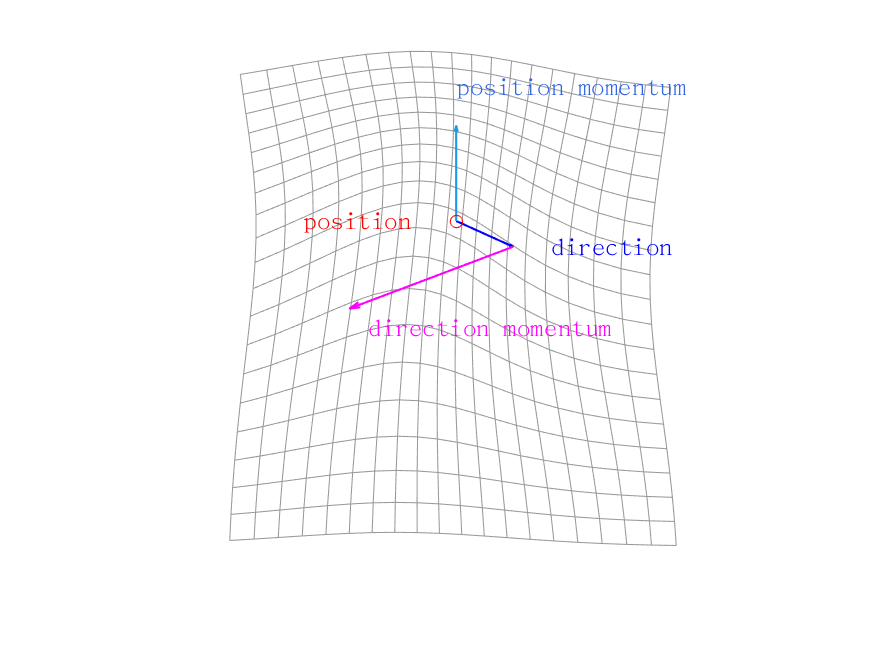}\\
    $t=0$ & $t=1/2$& $t=1$  \\
    \includegraphics[trim = 12mm 12mm 12mm 12mm ,clip,width=4cm]{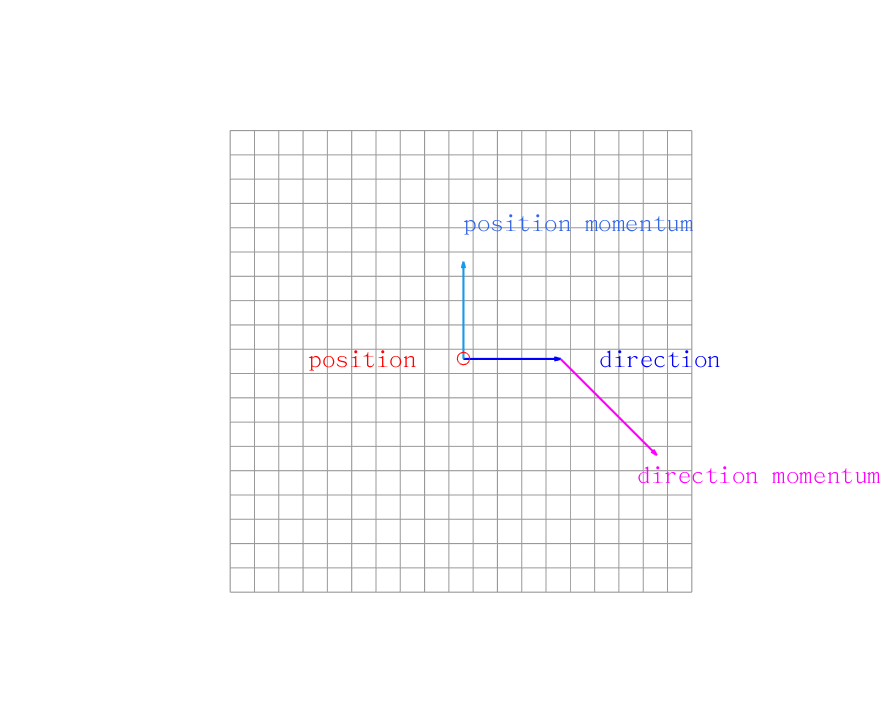} 
    &\includegraphics[trim = 12mm 12mm 12mm 12mm ,clip,width=4cm]{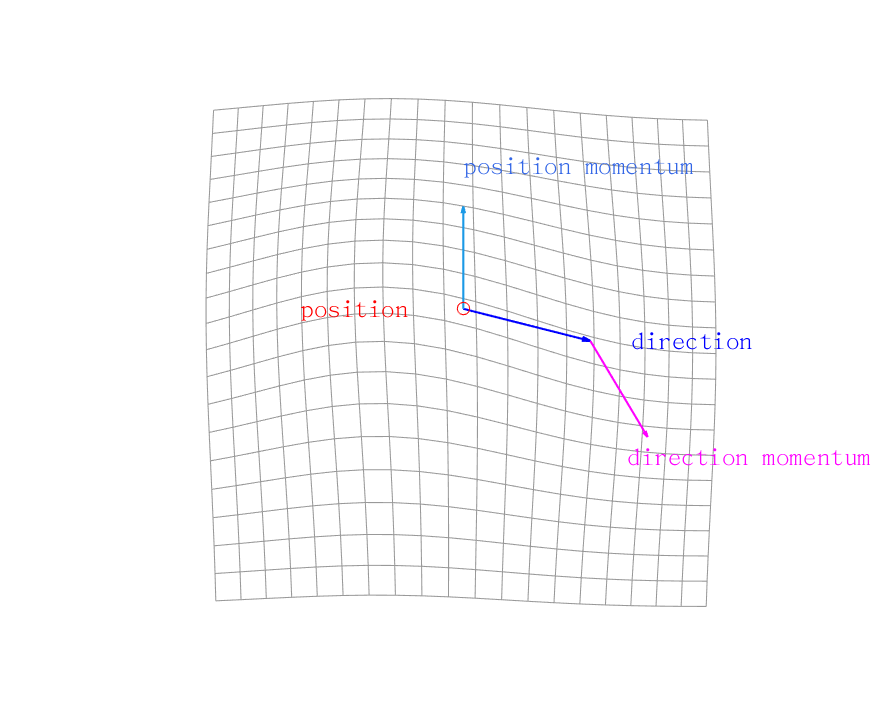} 
    &\includegraphics[trim = 12mm 12mm 12mm 12mm ,clip,width=4cm]{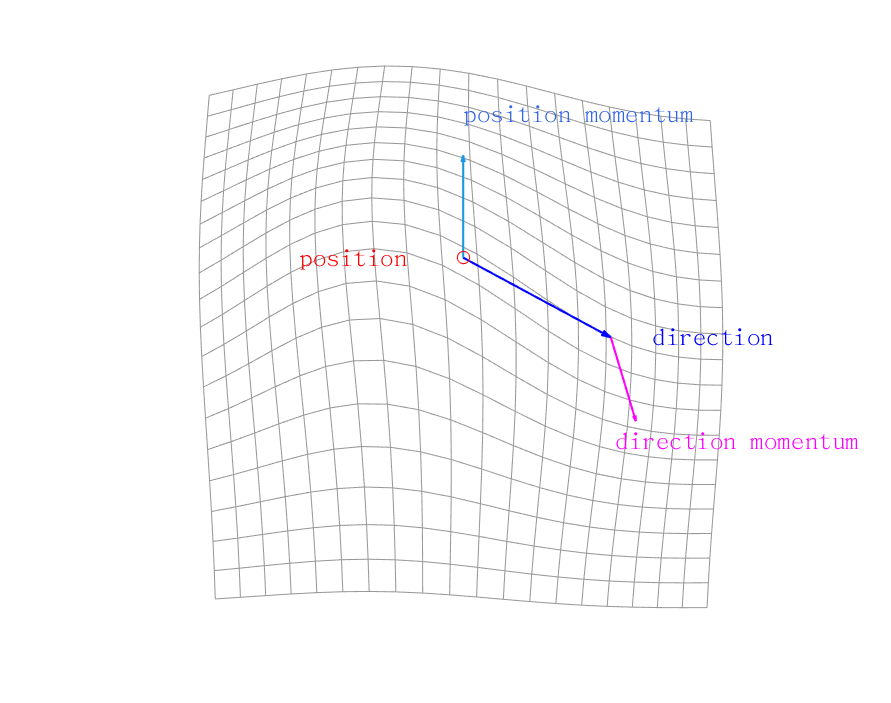}\\
    $t=0$ & $t=1/2$& $t=1$ 
    \end{tabular}
    \caption{Examples of geodesics in the pushforward action case.}
    \label{fig:pushforward}
\end{figure}

The proof is given in Appendix. Note that these equations can be also obtained in a more particular case as the geodesic equations on the tangent bundle of the space of landmarks, as derived for instance in \cite{Arguillere2015} (Section 3.5). In what follows, we will thus replace the system \eqref{fullforward} by \eqref{reduced_foward}. 

\begin{remark}
We point out that there are other conserved quantities in the previous system. In particular, it's easy to see that $\left\langle p_i^{(2)}, d_i \right\rangle$ is constant along geodesics since
\begin{align*}
\frac{d}{dt} \left\langle p_i^{(2)}, d_i \right\rangle = - \left\langle (d_{x_i}v_t)^T p_i^{(2)}, d_i \right\rangle + \left\langle p_i^{(i)}, d_{x_i} v_t(d_i) \right\rangle    =0.
\end{align*}
\end{remark}
Figure \ref{fig:pushforward} shows two geodesic trajectories of a single Dirac varifold for different initial momenta. In particular, we can again observe the effect of the directional momentum $p^{(2)}$ on the dynamics and resulting deformations. In addition to similar rotation effects as in the normalized action case, local contraction or expansion can be generated as well, depending precisely on the angle $\langle p_i^{(2)}, d_i \rangle$.

\section{Registration algorithm and implementation}
\label{sec:matching_algo}
We now turn to the issue of numerically solving the optimization problem \eqref{eq:matching_var}. We will follow the commonly used method for such problems called \textit{geodesic shooting} (cf \cite{Vialard2012b}). Indeed, from the developments of Section \ref{ssec:hamiltonian_dynamics}, we see that optimizing \eqref{eq:matching_var} with respect to vector fields $v$ can be done equivalently by restricting to geodesics and thus by optimizing over the initial momenta variables $p^{(1)}_0$ and $p^{(2)}_0$ that completely parametrize those geodesics through the Hamiltonian equations.  

\subsection{Computation of $E$}
Let a template and target discrete varifold be given as in Section \ref{ssec:optimal_control}. As mentioned above, we can rewrite the energy $E$ as a function of the initial momenta that we will denote $p^{(1)}_0=(p^{(1)}_i(0))$ and $p^{(2)}_0=(p^{(2)}_i(0))$:  
\begin{equation}
\label{eq:energy_discrete}
 E(p^{(1)}_0,p^{(2)}_0) = H_r(p_0,q_0) + \lambda \underbrace{\|\mu(1) - \tilde{\mu} \|_{W^*}^2}_{:=g(q(1))}
\end{equation}
where $q_0$ is the initial state, $\mu(1)$ is the varifold corresponding to the final time state $q(1)$ with $g(q(1))$ the resulting fidelity term between $\mu(1)$ and the target varifold, and $H_r$ is the \textit{reduced Hamiltonian} $H_r(p,q)=H(p,q,v)$ for the optimal $v$ given by \eqref{eq:v_normalized} or \eqref{reduced_vf} (note that $H_r(p(t),q(t))$ is conserved along solutions of the Hamiltonian systems thus giving the above expression of the energy). 

The expression of $H_r$ as well as the resulting reduced Hamiltonian equations can be obtained in all generality by plugging the expression of $v$ in the equations of Section \ref{ssec:optimal_control}. In our implementation, we actually restrict to the more particular case of radial scalar kernels for the vector fields in $V$, i.e we assume that $K(x,y)=h(|x-y|^2) I_n$. Then the reduced Hamiltonian for the normalized case becomes:
\begin{equation}
    H_r(p,q) = \frac{1}{2} \left\langle \overline{K}_q p, p \right\rangle,
\end{equation}
where $\overline{K}_q$ is a symmetric positive definite matrix which is defined as follows. Let

\begin{align*}
&H =(H)_{ik}= (h_{ki}) \\
&\overline{A} = (\overline{A})_{ik} = 2 \dot h_{ki} \langle x_k-x_i,d_k \rangle \\
&\overline{B} = (\overline{B})_{ik} = - \left[  4\ddot h_{ik} \langle x_k-x_i,d_i \rangle \langle  x_k-x_i, d_k \rangle + 2 \dot h_{ki} \langle d_i, d_k \rangle \right]
\end{align*}
with $h_{ki}$ being a shortcut for $h(|x_i-x_j|^2)$ and 

\begin{align*}
P_{d^{\perp}}(\cdot) =\left( \begin{array}{ccc}
I - d_1 \cdot d_1^T &  &\textrm{\huge{0}}  \\
   & \ddots & \\
\textrm{\huge{0}} & &  I - d_N \cdot d_N^T 
\end{array}\right)
\end{align*}
Then we define
\begin{align*}
\overline{K}_q := \left( \begin{array}{cc}
I_{Pd} &0  \\
0   & P_{d^{\perp}}  \\
\end{array}
 \right)^T
\left( \begin{array}{cc}
H \otimes I_{Pd}  & \overline{A} \otimes I_{Pd}\\
(\overline{A} \otimes I_{Pd})^T & \overline{B} \otimes I_{Pd} 
\end{array}
 \right)
\left( \begin{array}{cc}
I_{Pd} & 0  \\
0   & P_{d^{\perp}}  \\
\end{array}
 \right)  
\end{align*}
 where $\otimes$ denotes the Kronecker product. For the pushforward action case, we define $H$, $A$ and $B$ as in normalized action case with $d_i$ and $d_k$ replaced by $u_i$ and $u_k$, then
 \begin{equation}
    H_r(p,q) = \frac{1}{2} \left\langle K_q p, p \right\rangle,
\end{equation}
where 

\begin{align*}
K_q := \left( \begin{array}{cc}
H \otimes I_{Pd}  & A \otimes I_{Pd}\\
(A \otimes I_{Pd})^T & B \otimes I_{Pd} 
\end{array}
 \right).     
\end{align*}
 This gives us explicitly the first term of the energy in \eqref{eq:energy_discrete}.

Now, the time evolution of $q$ and $p$ can be also rewritten equivalently in reduced Hamiltonian form, which expressions are given in full for radial scalar kernels in the Appendix. We numerically integrate those differential systems using an RK4 scheme, which we experienced to be better-adapted to these systems than the simpler Euler midpoint integrator used in \cite{Charon2017}. Then, given initial momenta $p^{(1)}_0$ and $p^{(2)}_0$, integrating those equations forward in time produces the final state $q(1)$ and its corresponding varifold $\mu(1)$. It is then straightforward to evaluate the second term in \eqref{eq:energy_discrete} through the expression of the varifold norm \eqref{eq:metric_var}; in the pushforward case one only needs to apply the additional intermediate operation of converting state $q(1) = (x_i(1),u_i(1))$ into $(x_i(1),\overline{u_i}(1),|u_i|(1))$. We will discuss different choices of kernels for the varifold metric in the result section.  

\subsection{Computation of the gradient of $E$}
The second element we need is the gradient of the energy with respect to the momenta. The first term being directly a function of $p_0$, it can be differentiated easily and gives the following gradient: 
\begin{align*}
    \nabla_{p_0} H_r(p_0,q_0) &= \overline{K}_q p_0 \ \ \text{(normalized)} \\
    \nabla_{p_0} H_r(p_0,q_0) &= K_q p_0 \ \ \text{(pushforward)}
\end{align*}

The fidelity term $g(q(1))$ in \eqref{eq:energy_discrete}, however, is a function of the final state $q(1)$ which is in turn a function of the momenta through the Hamiltonian system of equations. The computation of the gradient is therefore more involved due to the complicated dependency of $q(1)$ in $p_0$. The standard approach for optimal control problems of this form (cf for example \cite{Vialard2012b} or \cite{arguillere14:_shape}) is to introduce the \textit{adjoint Hamiltonian system}: 
$$ \overset{\bold{\ldotp}}{Z}(t) = d(\partial_p H_r , \partial_q H_r)^{T} Z(t) $$ 
with $Z = (\tilde{q}, \tilde{p})^{T}$ the vector of the adjoint variables. Then, as detailed in \cite{arguillere14:_shape}, the gradient of $g(q(1))$ with respect to $p_0$ is given by $\tilde{p}(0)$ where $(\tilde{q}(t), \tilde{p}(t))^{T}$ is the solution of the adjoint system integrated backward in time with $\tilde{q}(1) = \nabla g(q(1))$ and $\tilde{p}(1) = 0$. 

For the particular Hamiltonian equations considered here, the adjoint system is tedious to derive and to implement. We simply avoid that by approximating the differentials appearing in the adjoint system by finite difference of the forward Hamiltonian equations, following the suggestion of \cite{arguillere14:_shape} (Section 4.1) which we refer to for details. Note that another possibility would be to take advantage of automatic differentiation methods, as used recently for some LDDMM registration problems by the authors of \cite{Kuhnel2017}. 

Lastly, the end time condition $\nabla g(q(1))$ in the previous adjoint system is computed by direct differentiation of the varifold norm \eqref{eq:metric_var} with respect to the final state variables. This is actually more direct than in previous works like \cite{Charon2,Charon2017} where the gradients are computed with respect to the positions of vertices of the underlying mesh. Here, we have specifically, for the normalized model:  
\begin{align*}
    \partial_{x_i} g(q(1)) &= 2 \sum_{j=1}^{P} 2 r_i r_j \rho'(|x_i(1)-x_j(1)|^2) \gamma(\langle d_i(1), d_j(1) \rangle) .(x_i(1)-x_j(1)) \ \ -2\ldots \\
    \partial_{d_i} g(q(1)) &= 2 \sum_{j=1}^{P} r_i r_j \rho(|x_i(1)-x_j(1)|^2) \gamma'(\langle d_i(1), d_j(1) \rangle). d_j(1) \ \ -2\ldots
\end{align*}
where the $\ldots$ denote a similar term for the differential of the cross inner product $\langle \mu(1),\tilde{\mu} \rangle_{W^*}$. In the pushforward case with state variables $(x_i(1),u_i(1))$, we first compute $d_i(1)=u_i(1)/|u_i(1)|$ and $r_i(1)=|u_i(1)|$ and obtain $\partial_{x_i} g(q(1))$ with the same expression as above while $\partial_{u_i} g(q(1))$ is given by a simple chain rule. 

Finally, with the above notations, the gradient of $E$ writes:
\begin{equation}
    \label{eq:gradient_E}
    \nabla_{p_0} E = \overline{K}_{q} p_0 + \lambda \tilde{p}(0)
\end{equation}
respectively $\nabla_{p_0} E = K_{q} p_0 + \lambda \tilde{p}(0)$ in the pushforward case.

\subsection{Gradient descent algorithm}
The solution to the minimization of \eqref{eq:energy_discrete} is then computed by gradient descent on $p_0 = \left( p^{(1)}_{0,i},p^{(2)}_{0,i} \right)_{i=1,\ldots,P}$. Note that this is a non-convex optimization problem. Until convergence, each iteration consists of the following steps:
\vskip2ex
\noindent (1) Given the current estimate of $p_0$, integrate the Hamiltonian equations forward in time to obtain $q(1)$.
\vskip1ex
\noindent (2) Compute the gradient $\nabla g(q(1))$.
\vskip1ex
\noindent (3) Integrate the adjoint Hamiltonian system backward in time to obtain $\nabla_{p_0} E$. 
\vskip1ex
\noindent (4) Update $p_0$: we use two separate update steps for the spatial and directional momentum which are selected, at each iteration, using a rough space search approach leading to the lowest value of $E$.    

\section{Results}
\label{sec:results}
We now present a few results of registration using the previous algorithm on simple and synthetic examples. Our implementation equally supports objects in 2D or 3D, we will however focus on examples in $\R^2$ here simply to allow for an easier visualization and interpretation of the results.   

\subsection{Curve registration}
\begin{figure}
    \centering
    \begin{tabular}{cccc}
    \includegraphics[trim = 20mm 20mm 20mm 20mm ,clip,width=3.3cm]{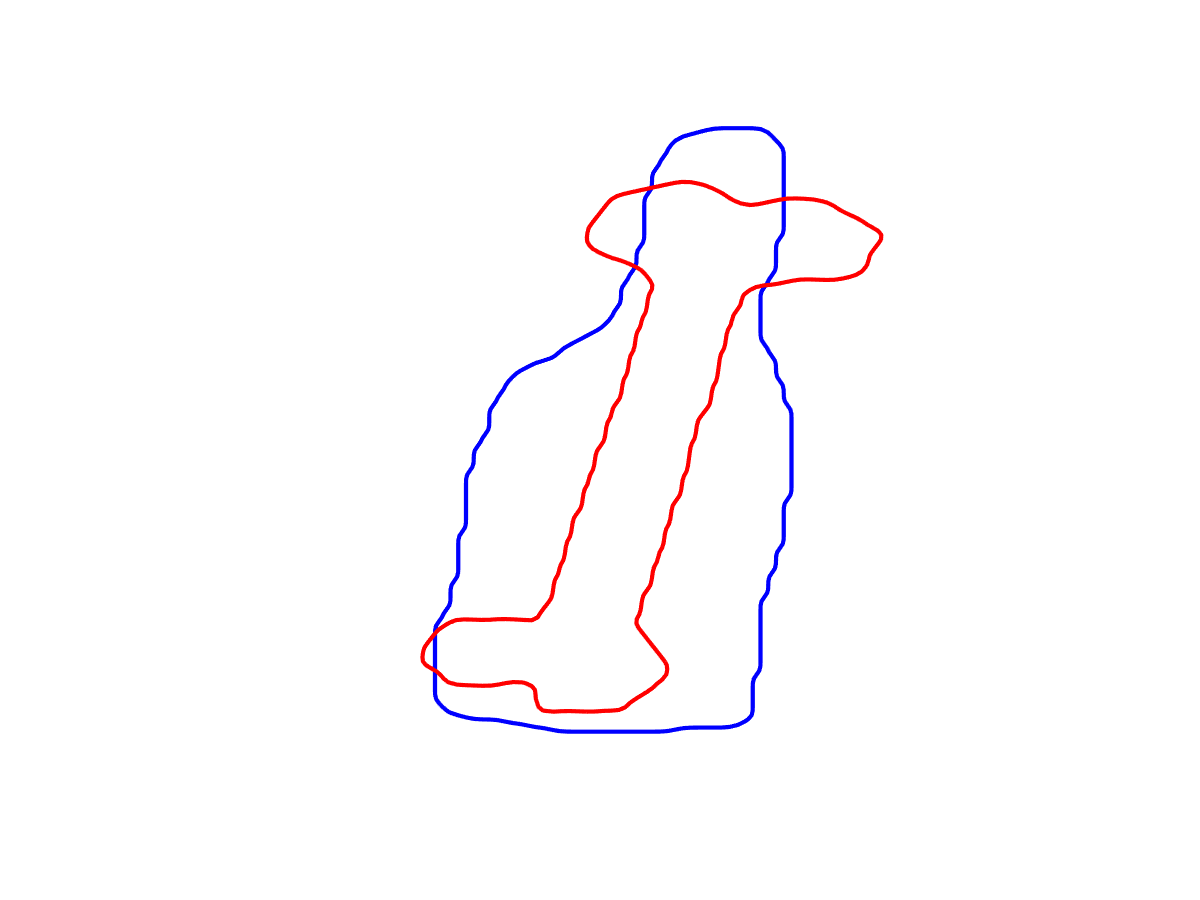} 
    &\includegraphics[trim = 20mm 20mm 20mm 20mm ,clip,width=3.3cm]{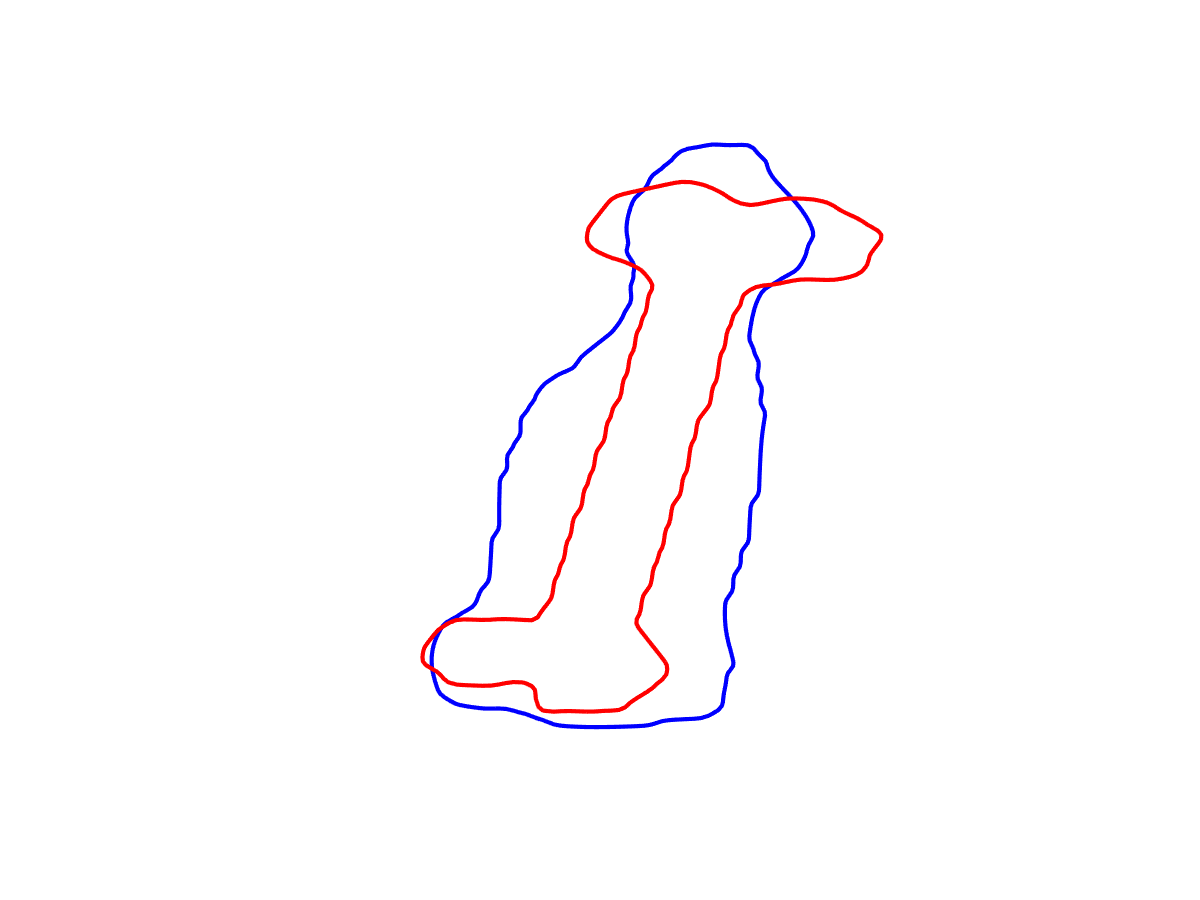} 
    &\includegraphics[trim = 20mm 20mm 20mm 20mm ,clip,width=3.3cm]{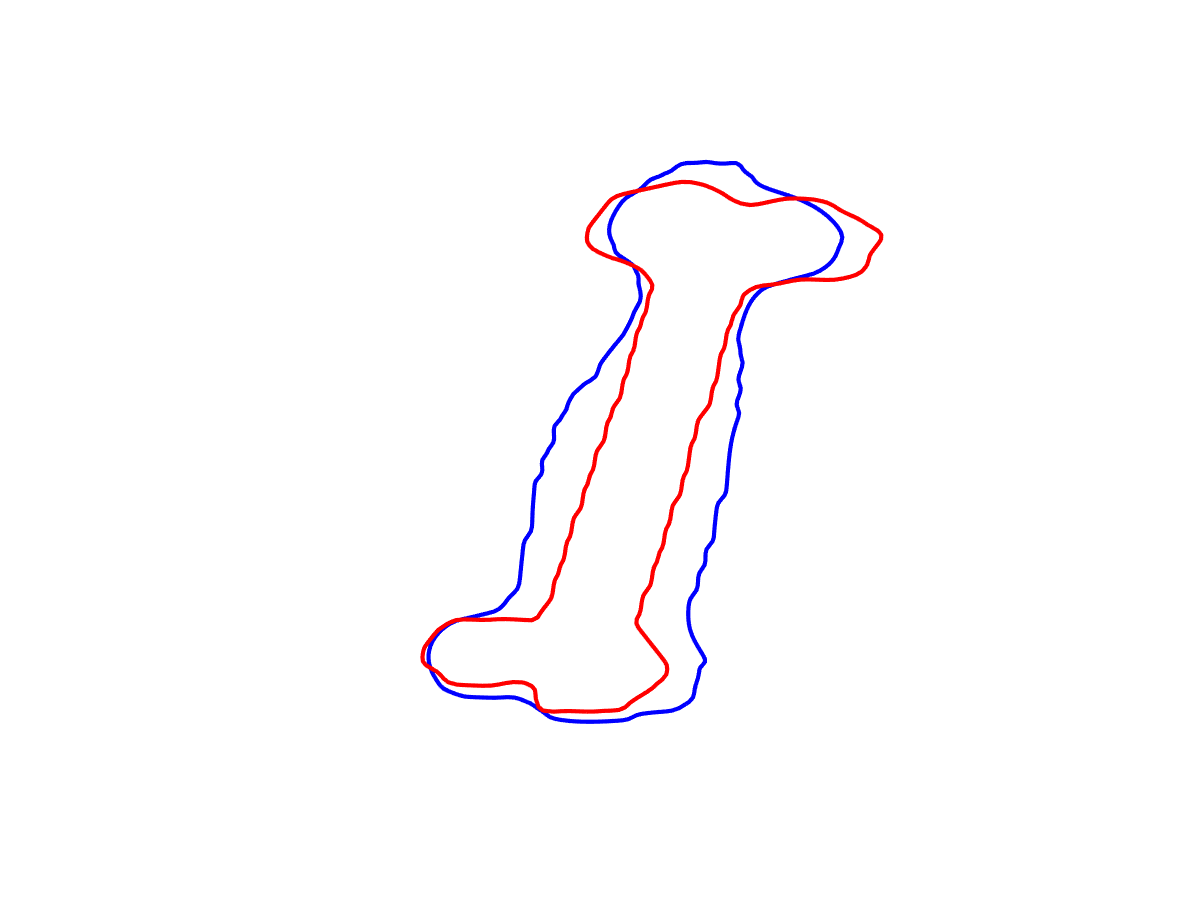}
    &\includegraphics[trim = 20mm 20mm 20mm 20mm ,clip,width=3.3cm]{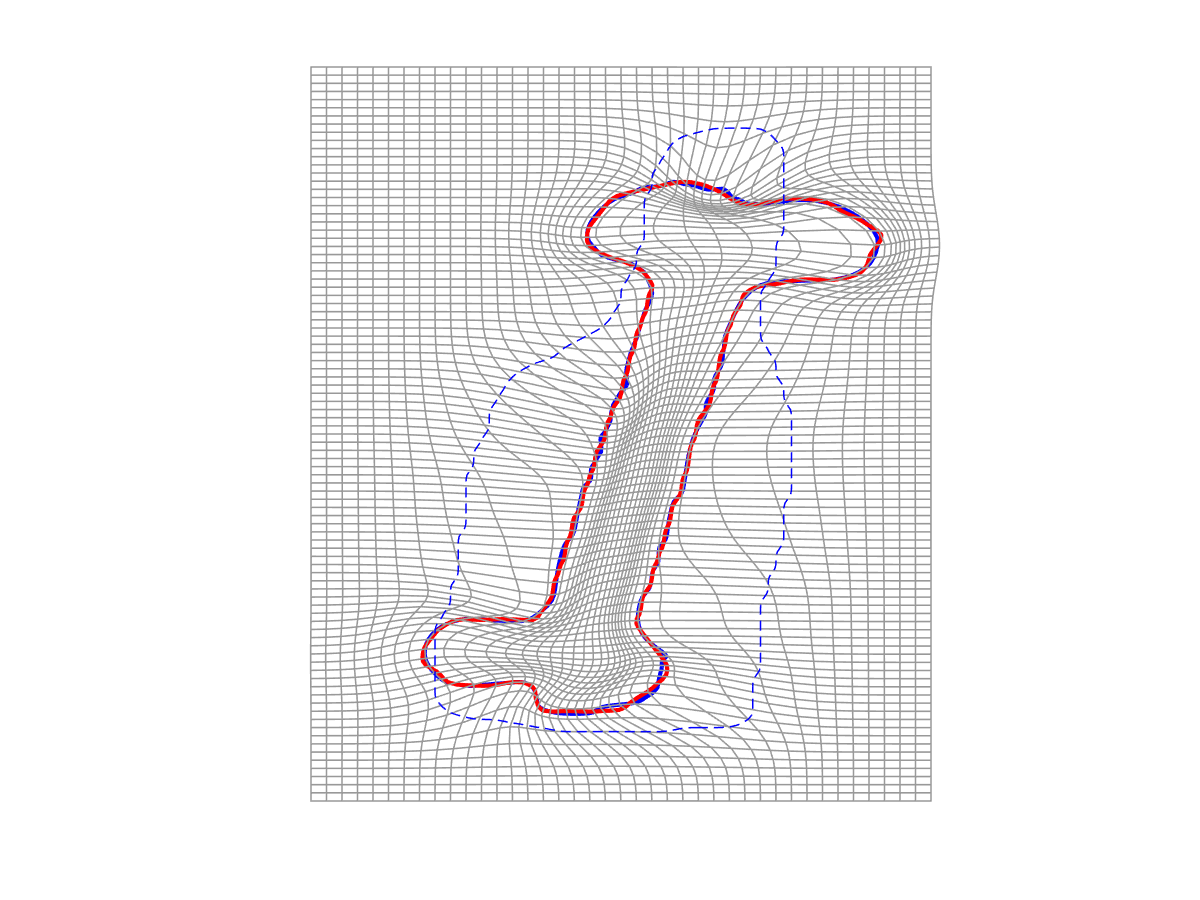} \\
    $t=0$ & $t=1/3$ & $t=2/3$  & $t=1$ \\
    \includegraphics[trim = 20mm 20mm 20mm 20mm ,clip,width=3.3cm]{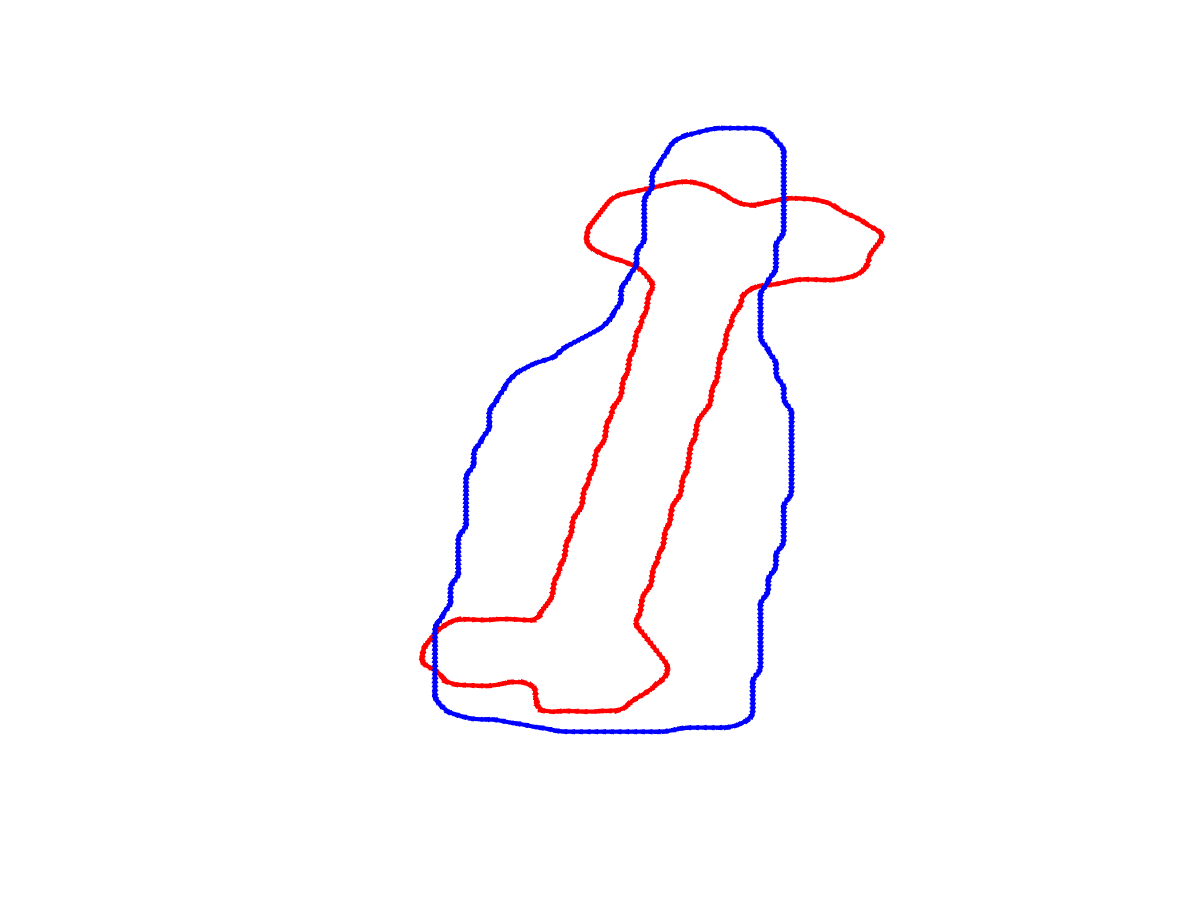} 
    &\includegraphics[trim = 20mm 20mm 20mm 20mm ,clip,width=3.3cm]{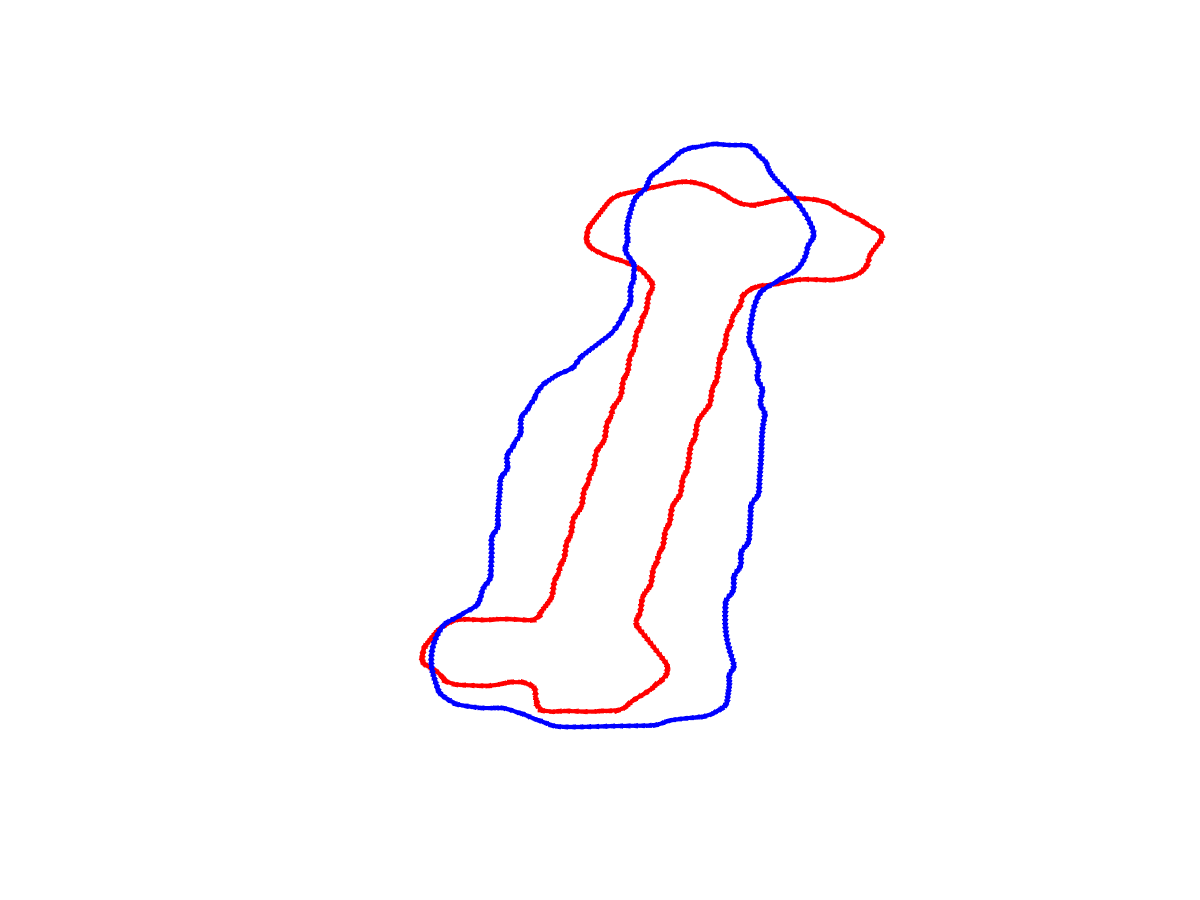} 
    &\includegraphics[trim = 20mm 20mm 20mm 20mm ,clip,width=3.3cm]{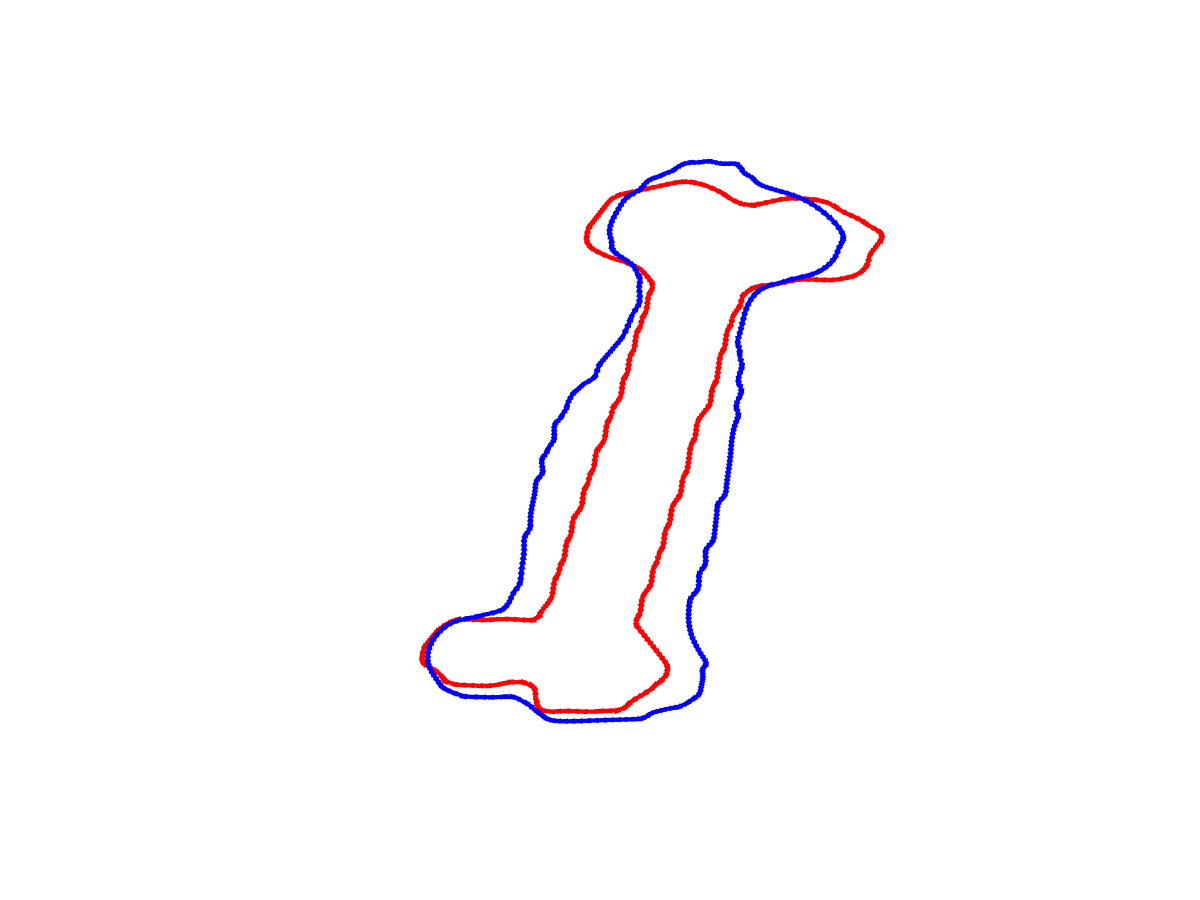}
    &\includegraphics[trim = 20mm 20mm 20mm 20mm ,clip,width=3.3cm]{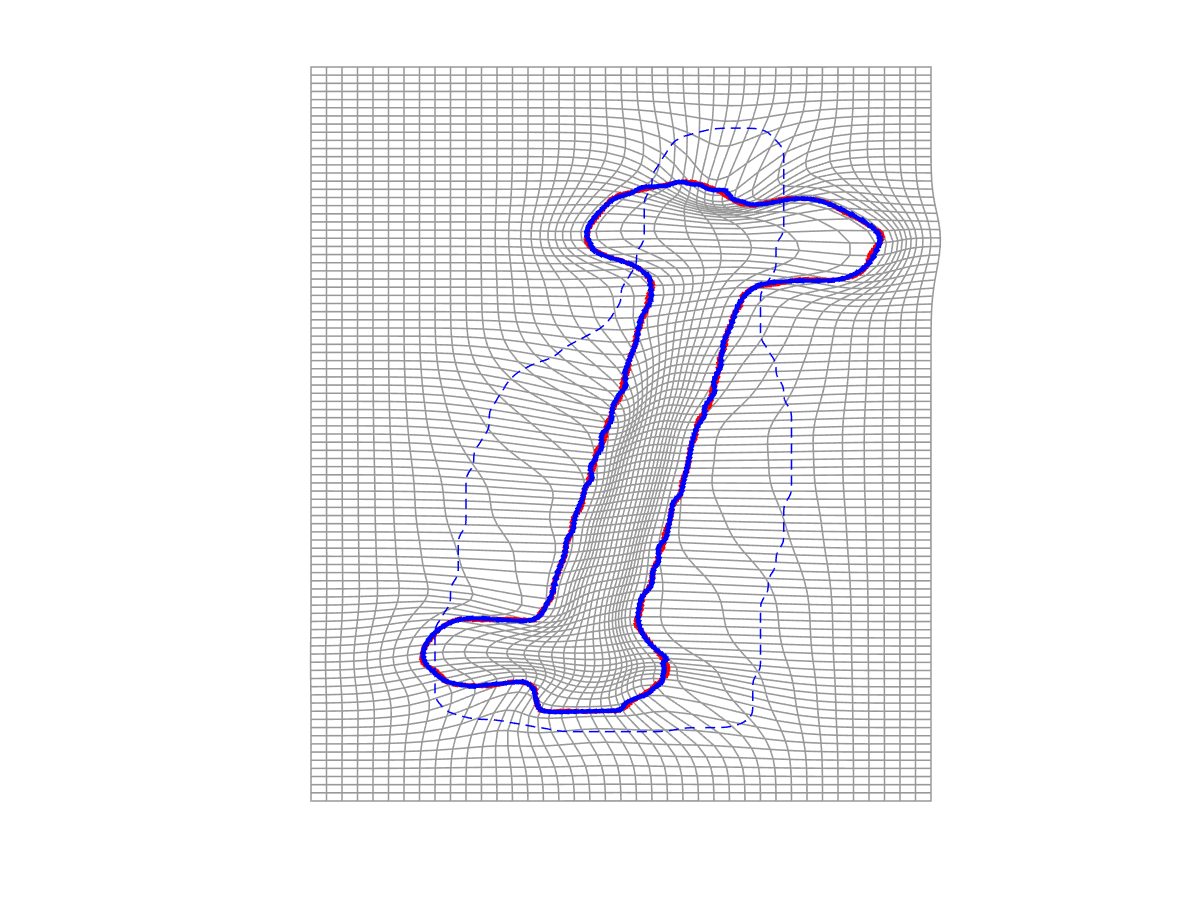}     
    \end{tabular}
    \includegraphics[width=5cm]{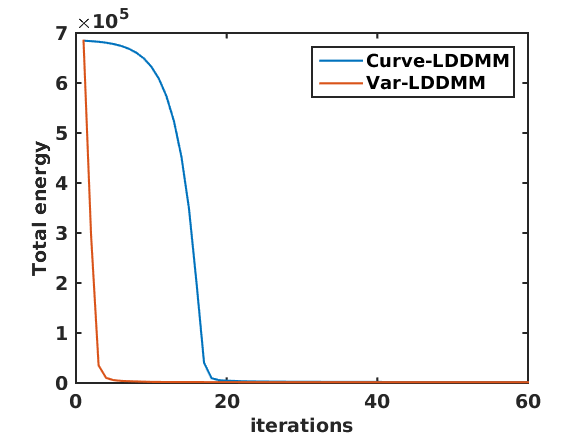}
    \caption{Curve registration using point-mesh LDDMM (1st row) and our proposed discrete varifold LDDMM (2nd row). On the last row is shown the evolution of the total energy across the iterations for both algorithms.}
    \label{fig:bottle}
\end{figure}
We begin with a toy example of standard curve matching to compare the result and performance of our discrete varifold LDDMM registration algorithm with the state-of-the-art LDDMM approach for curves such as the implementations of \cite{Glaunes2008,Charon2017}. The former methods share a very similar formulation to \eqref{eq:matching_var} and also make use of varifold metrics as fidelity terms, the essential difference being that the state of the optimal control problem is there the set of vertices of the deformed template curve which is only converted to a varifold for the evaluation of the fidelity term at each iteration. But the dynamics of geodesics still correspond to usual point set deformation under the LDDMM model.   

We consider here the pushforward model for the action of diffeomorphisms on discrete varifolds that we have seen is compatible with the action of diffeomorphisms on curves. In this case, the two formulation and optimization problems for curve registration are theoretically equivalent up to discretization precision. We verify it with the example of Figure \ref{fig:bottle} for which both algorithms are applied with the same deformation kernel, varifold metric and optimization scheme. Note that in our approach, template and target curves are first (and only once at the beginning) converted to their discrete varifold representations as explained in Section \ref{sec:discrete_varifolds}. 

As we can see, the resulting geodesics and deformations are consistent between the two methods. This is also corroborated by the very similar values of the energy at convergence. Interestingly however, although each iteration in our model is arguably more expansive numerically compared to standard curve-LDDMM due to the increased complexity of the Hamiltonian equations, the algorithm converges in a significantly lesser number of iterations. Whether this observation generalizes to other examples or other optimization methods will obviously require more careful examination in future work. 

\subsection{Registration of directional sets}
We now turn to examples that are more specific to the framework of discrete varifolds. 

\subsubsection*{Choice of the varifold metric} 
First, we examine more closely the effect of the metric $\|\cdot\|_{W^*}$ on the registration of discrete varifolds. The framework we propose can indeed support many choices for the kernel functions $\rho$ and $\gamma$ that define fidelity metrics $\|\cdot\|_{W^*}$ with possibly very different properties. This has been already analyzed quite extensively in \cite{Charon2017} but only in the situation where varifolds associated to a curve or a surface. We consider here the same examples of kernels and briefly discuss what are the specific effects to expect when matching more general varifolds in $\mathcal{D}$ which may involve several orientation vectors at a given position.

\begin{figure}
    \centering
    \begin{tabular}{ccccc}
    \rotatebox{90}{\phantom{aaaaa} Binet}
    &\includegraphics[trim = 8mm 8mm 8mm 8mm ,clip,width=3cm]{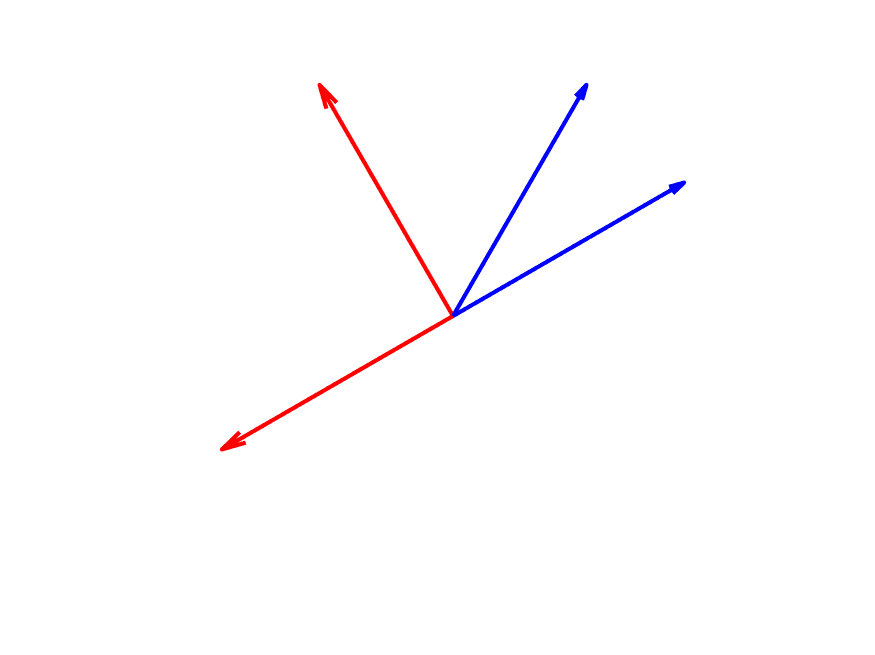} 
    &\includegraphics[trim = 8mm 8mm 8mm 8mm ,clip,width=3cm]{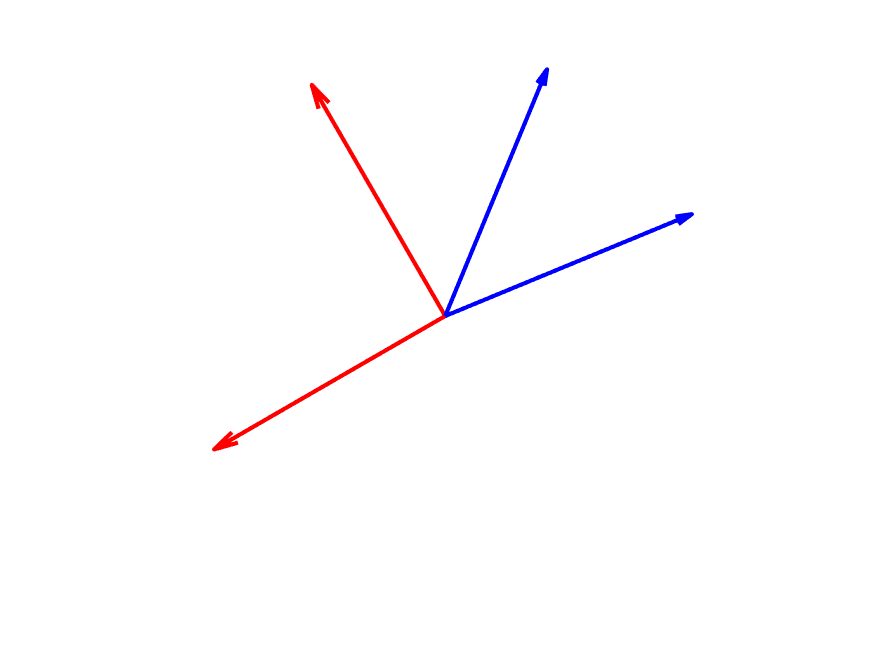} 
    &\includegraphics[trim = 8mm 8mm 8mm 8mm,clip,width=3cm]{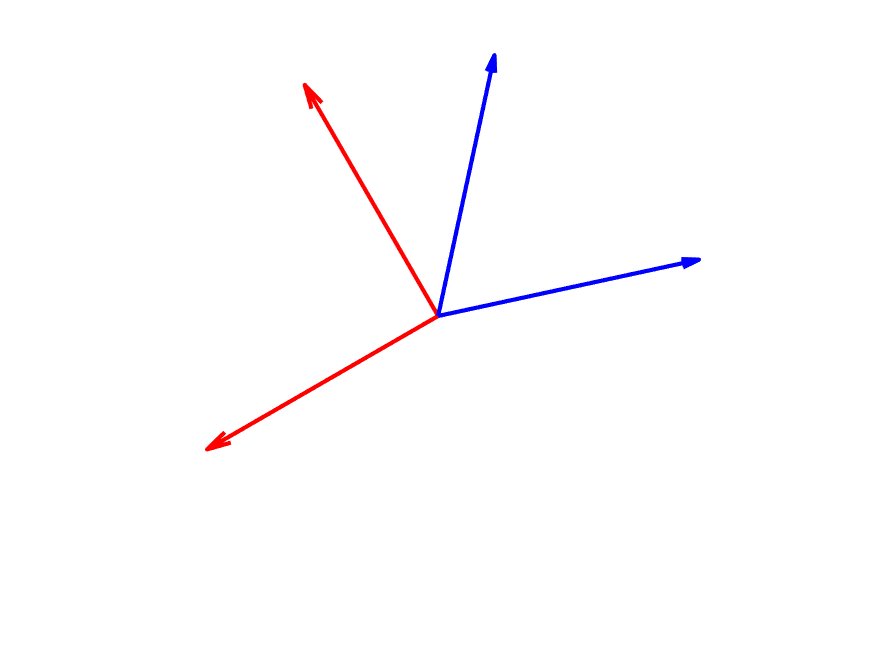}
    &\includegraphics[trim = 8mm 8mm 8mm 8mm ,clip,width=3cm]{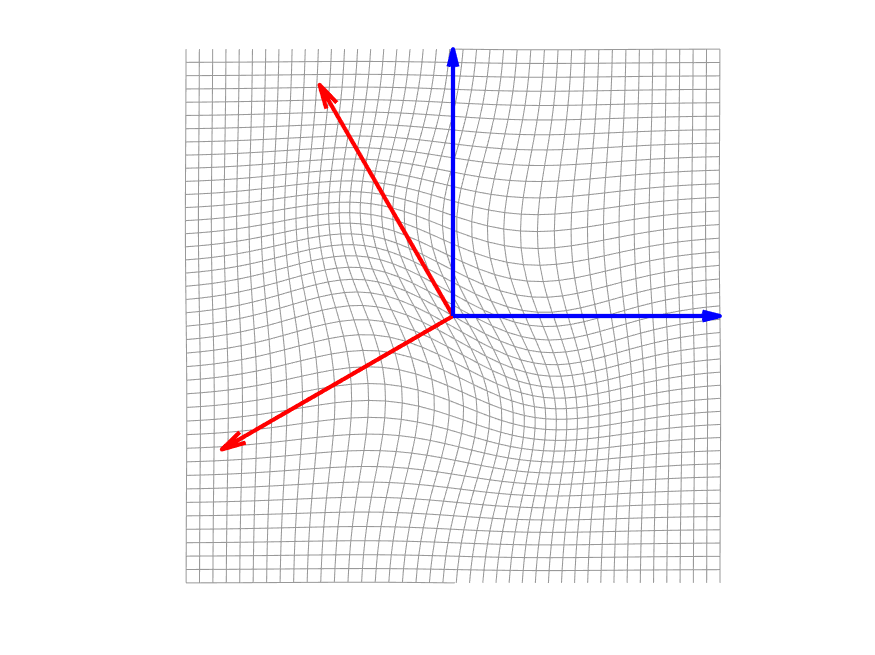} \\ 
    \rotatebox{90}{\phantom{aa} Unor. Gaussian}
    &\includegraphics[trim = 8mm 8mm 8mm 8mm ,clip,width=3cm]{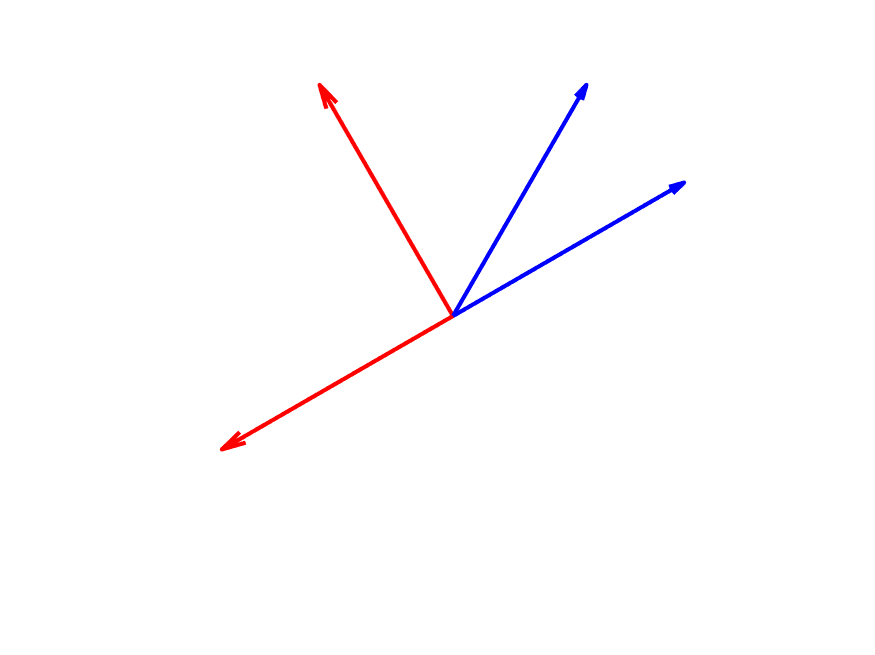} 
    &\includegraphics[trim = 8mm 8mm 8mm 8mm ,clip,width=3cm]{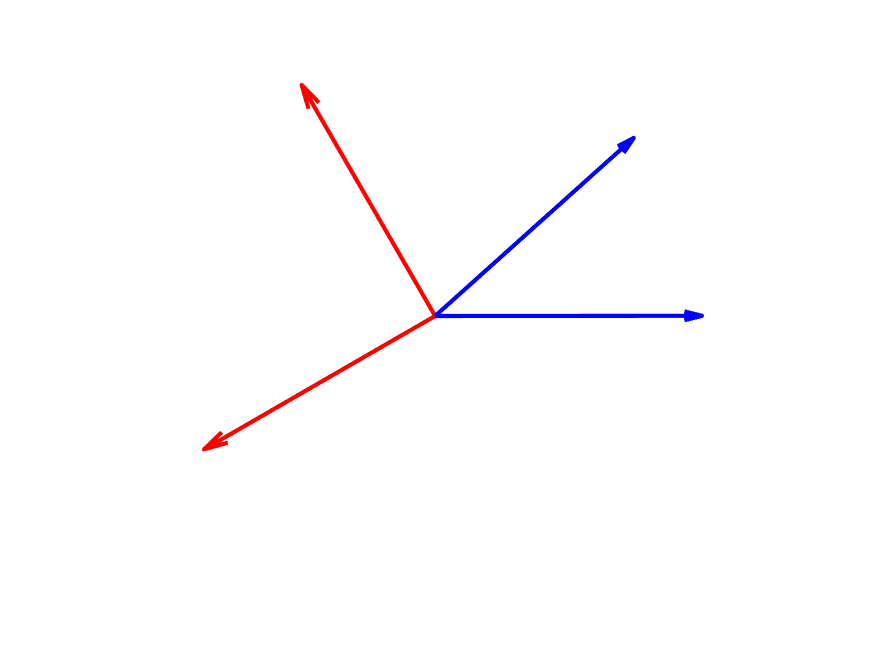} 
    &\includegraphics[trim = 8mm 8mm 8mm 8mm,clip,width=3cm]{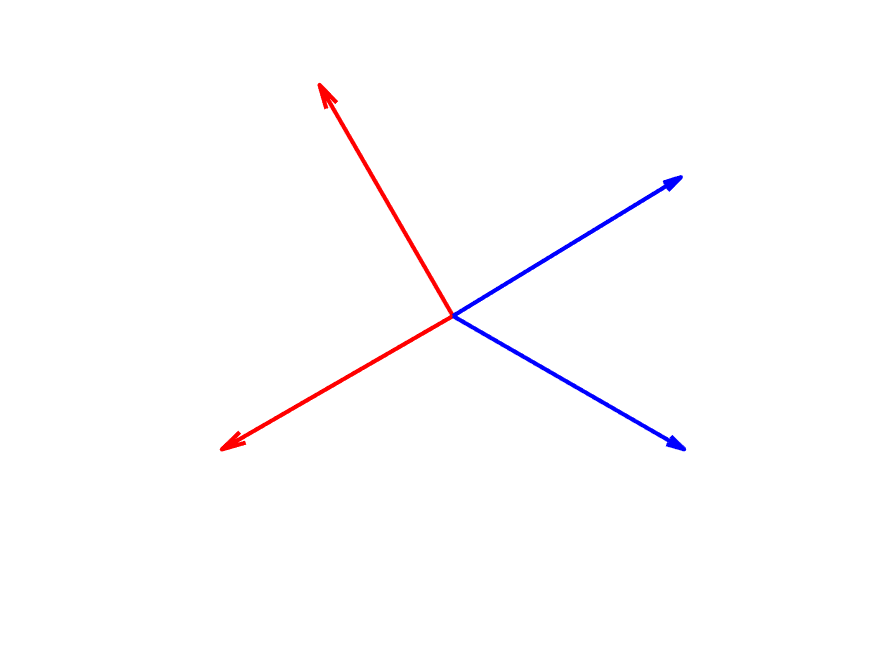}
    &\includegraphics[trim = 8mm 8mm 8mm 8mm ,clip,width=3cm]{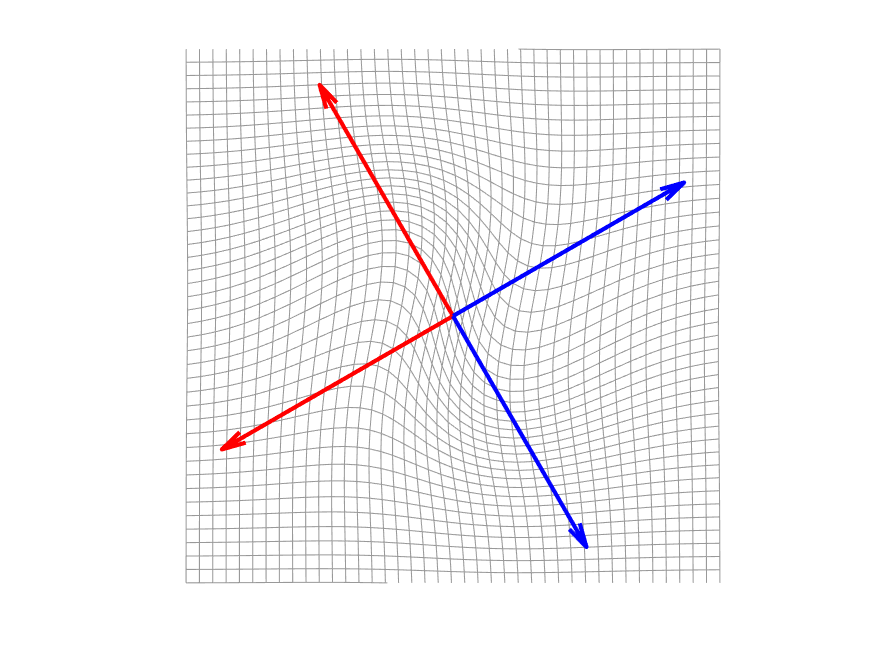} \\
    \rotatebox{90}{Or. Gaussian/Linear}
    &\includegraphics[trim = 8mm 8mm 8mm 8mm ,clip,width=3cm]{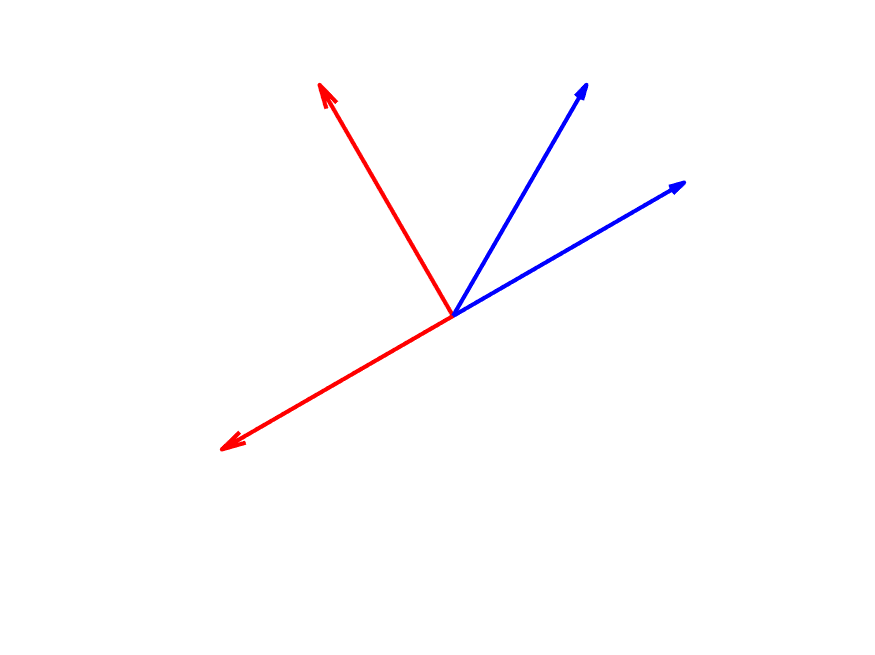} 
    &\includegraphics[trim = 8mm 8mm 8mm 8mm ,clip,width=3cm]{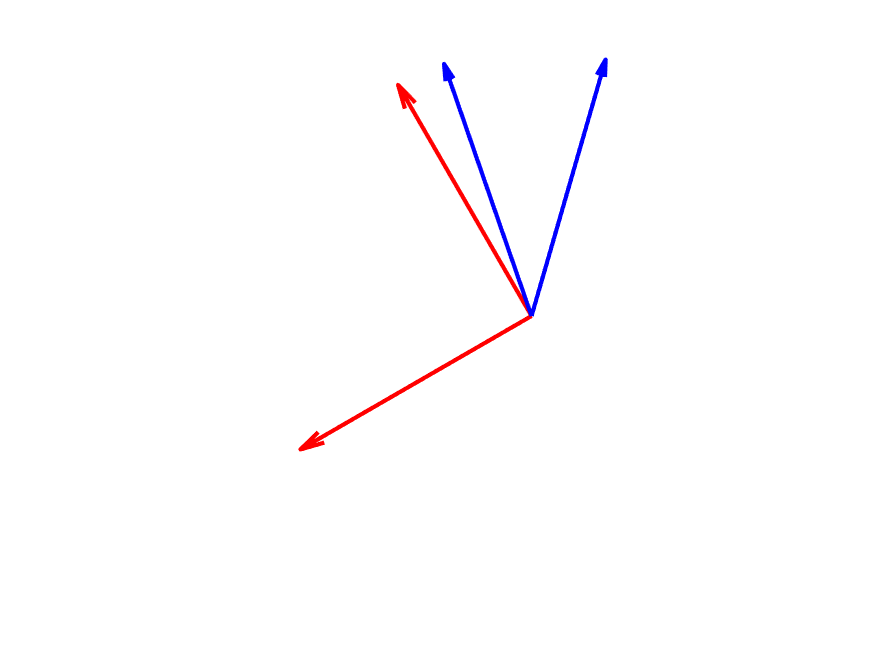} 
    &\includegraphics[trim = 8mm 8mm 8mm 8mm,clip,width=3cm]{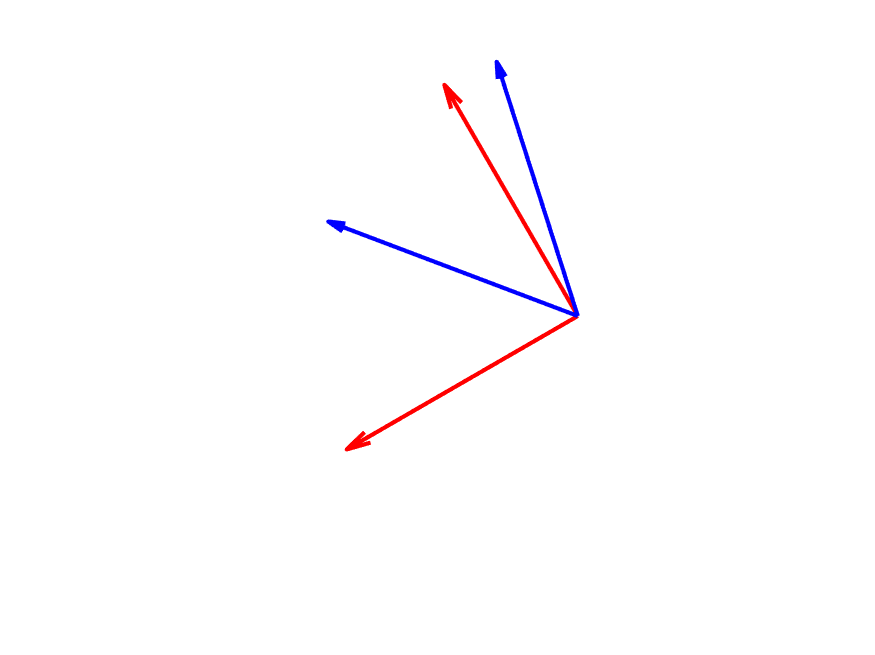}
    &\includegraphics[trim = 8mm 8mm 8mm 8mm ,clip,width=3cm]{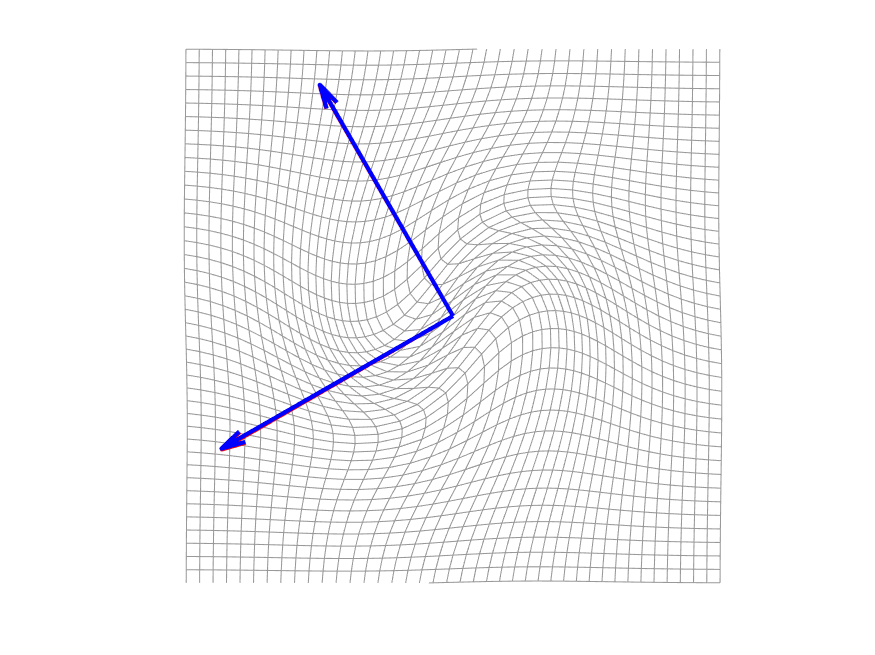} \\
     &$t=0$ & $t=1/3$ & $t=2/3$  & $t=1$
    \end{tabular}
    \caption{Matching of pairs of Dirac varifolds (template is in blue and target in red) under the normalized action with different choices of kernels: Binet on the first row, unoriented Gaussian ($\sigma_s =1$) on the second and oriented Gaussian ($\sigma_s =2$) on the last one. The linear kernel leads to the same result as the former in that particular case.}
    \label{fig:2Diracs}
\end{figure}

The results of Propositions \ref{prop:var_metric2} and \ref{prop:var_metric3} hold under the assumption that the kernel defined by $\rho$ is a $C_0$-universal kernel on $\R^n$, which restricts the possible choices to a few known classes (cf \cite{Carmeli2010} for a thorough analysis). Here, we will focus on the class of Gaussian kernels given by $\rho(|x-x'|^2) = e^{-\frac{|x-x'|^2}{\sigma^2}}$ with a width parameter $\sigma>0$ that essentially provides a notion of spatial scale sensitivity to the metric, and which must be adapted to the intrinsic sizes of shapes in each example.  

In combination with $\rho$, as in \cite{Charon2017}, we introduce the following four kernels on $\S^{n-1}$:
\begin{itemize}
\item[$\bullet$] $\gamma(\langle d,d'\rangle) = \langle d,d'\rangle$ (linear kernel): this choice is related to the particular subclass of currents \cite{Glaunes2008}. In that case, the resulting $\|\cdot\|_{W^*}$ is clearly only a pseudo-metric on $\mathcal{D}$ since the linearity implies that in $W^*$: $\delta_{(x,-d)}=-\delta_{(x,d)}$ and for any $d_1\neq -d_2$, $\delta_{(x,d_1)} + \delta_{(x,d_2)} = |d_1 + d_2| \delta_{\left(x,\frac{d_1+d_2}{|d_1+d_2|}\right)}$. However, we still obtain a metric on the subspace $\mathring{\mathcal{D}}$ thanks to Proposition \ref{prop:var_metric2}.  

\item[$\bullet$] $\gamma(\langle d,d'\rangle) = \langle d,d'\rangle^2$ (Binet kernel): $\gamma$ being an even function, as discussed in Section \ref{sec:discrete_varifolds}, the resulting metric on $W^*$ is invariant to the orientation of direction vectors. According to Proposition \ref{prop:var_metric3}, we then have a distance on $\mathring{\mathcal{D}}$ modulo the orientation. Note however that with this particular choice, one does not obtain a metric (but only a pseudo-metric) on $\mathcal{D}$ modulo the orientation, as we will illustrate in the examples below.        

\item[$\bullet$] $\gamma(\langle d,d'\rangle) = e^{-\frac{2}{\sigma_s^2}(1-\langle d,d'\rangle^2)}$ (unoriented Gaussian kernel): this is another example of orientation-invariant kernel considered in \cite{Charon2} corresponding to a particular construction of Gaussian kernels on the projective space. In contrast with Binet kernel, it does induce a metric on $\mathcal{D}$ modulo orientation.

\item[$\bullet$] $\gamma(\langle d,d'\rangle) = e^{-\frac{2}{\sigma_s^2}(1-\langle d,d'\rangle)}$ (oriented Gaussian kernel): this kernel is the restriction of the standard Gaussian kernel on $\R^n$ to the sphere $\S^{n-1}$. As such, it can be shown to be $C_0$-universal on $\S^{n-1}$ and thus, from Proposition \ref{prop:var_metric1}, lead to a metric on the entire space $\mathcal{D}$.
\end{itemize}

\begin{figure}
    \centering
    \begin{tabular}{ccccc}
    \rotatebox{90}{\phantom{aaaa} Linear}
    &\includegraphics[trim = 8mm 8mm 8mm 8mm ,clip,width=3cm]{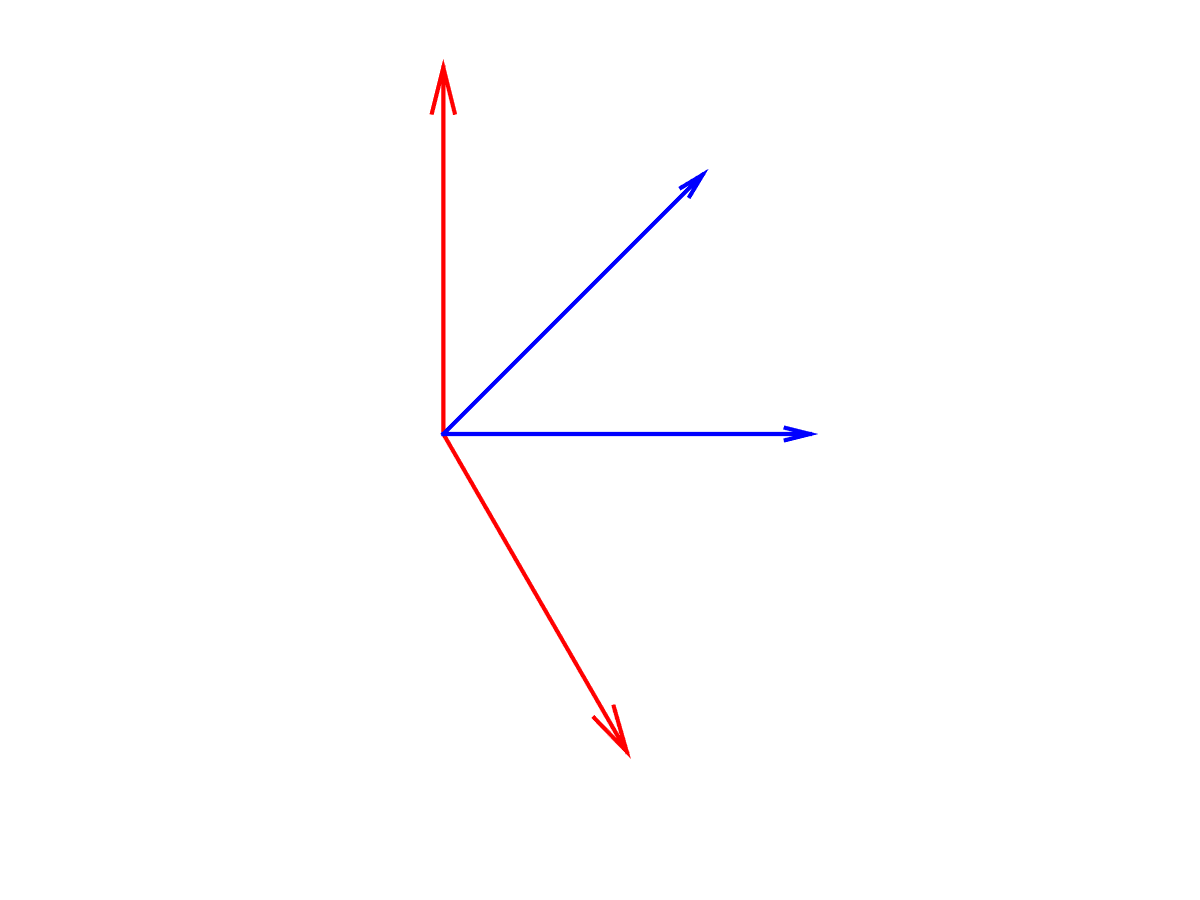} 
    &\includegraphics[trim = 8mm 8mm 8mm 8mm ,clip,width=3cm]{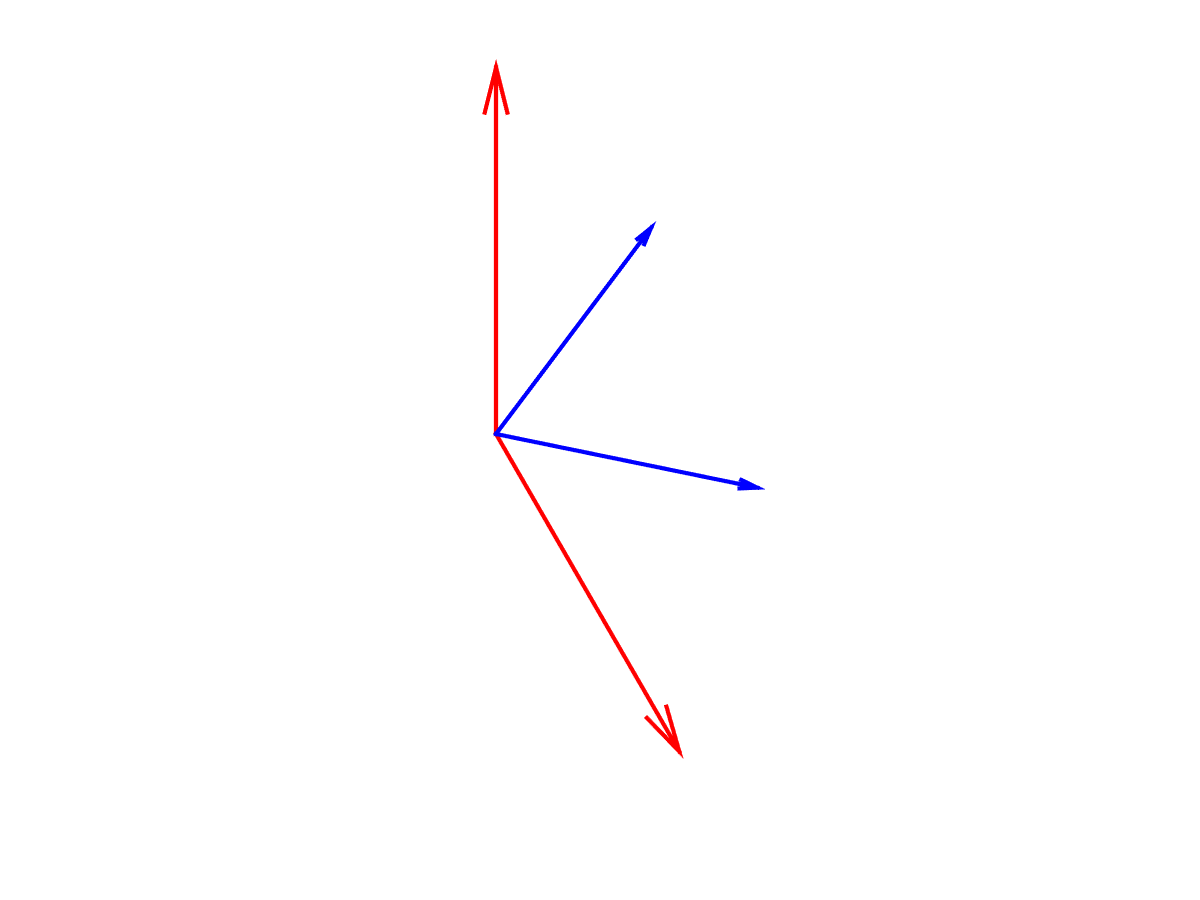} 
    &\includegraphics[trim = 8mm 8mm 8mm 8mm,clip,width=3cm]{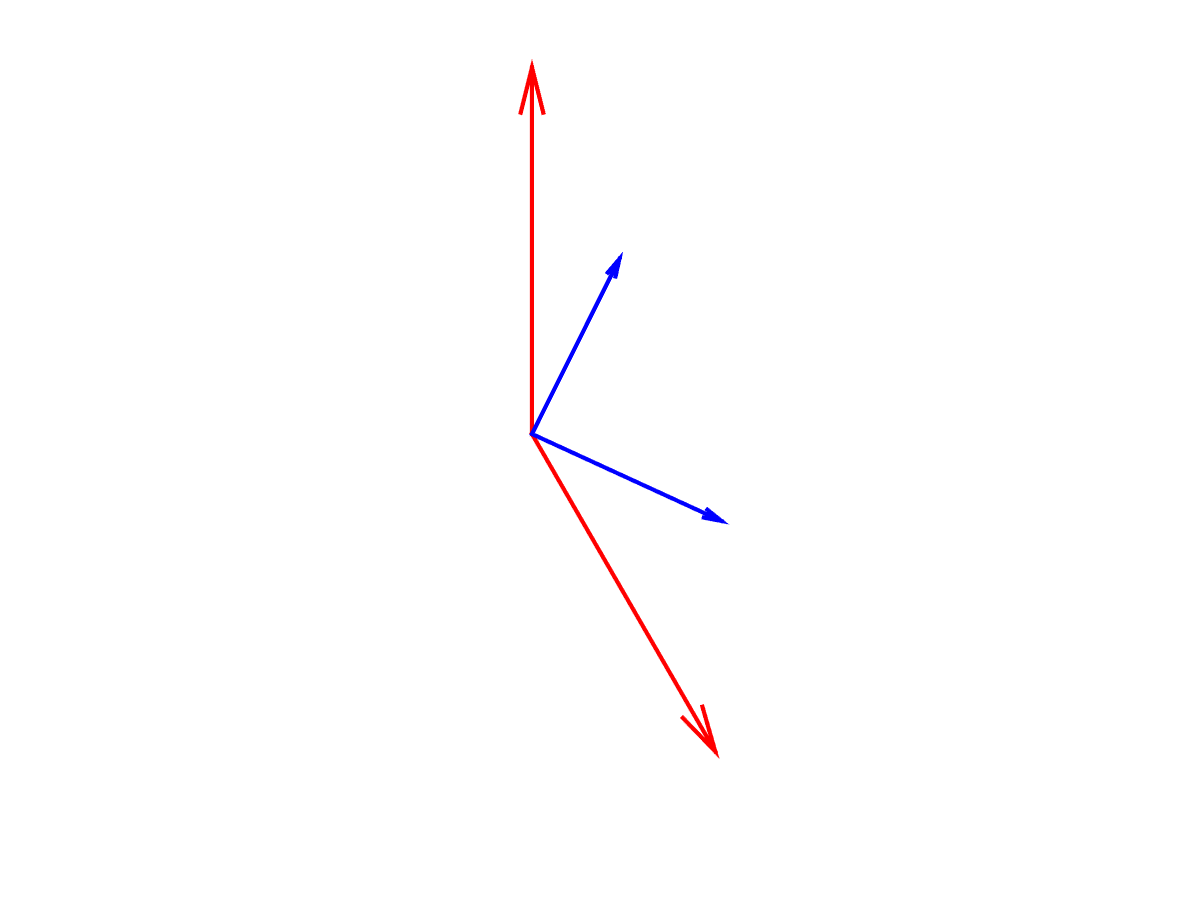}
    &\includegraphics[trim = 8mm 8mm 8mm 8mm ,clip,width=3cm]{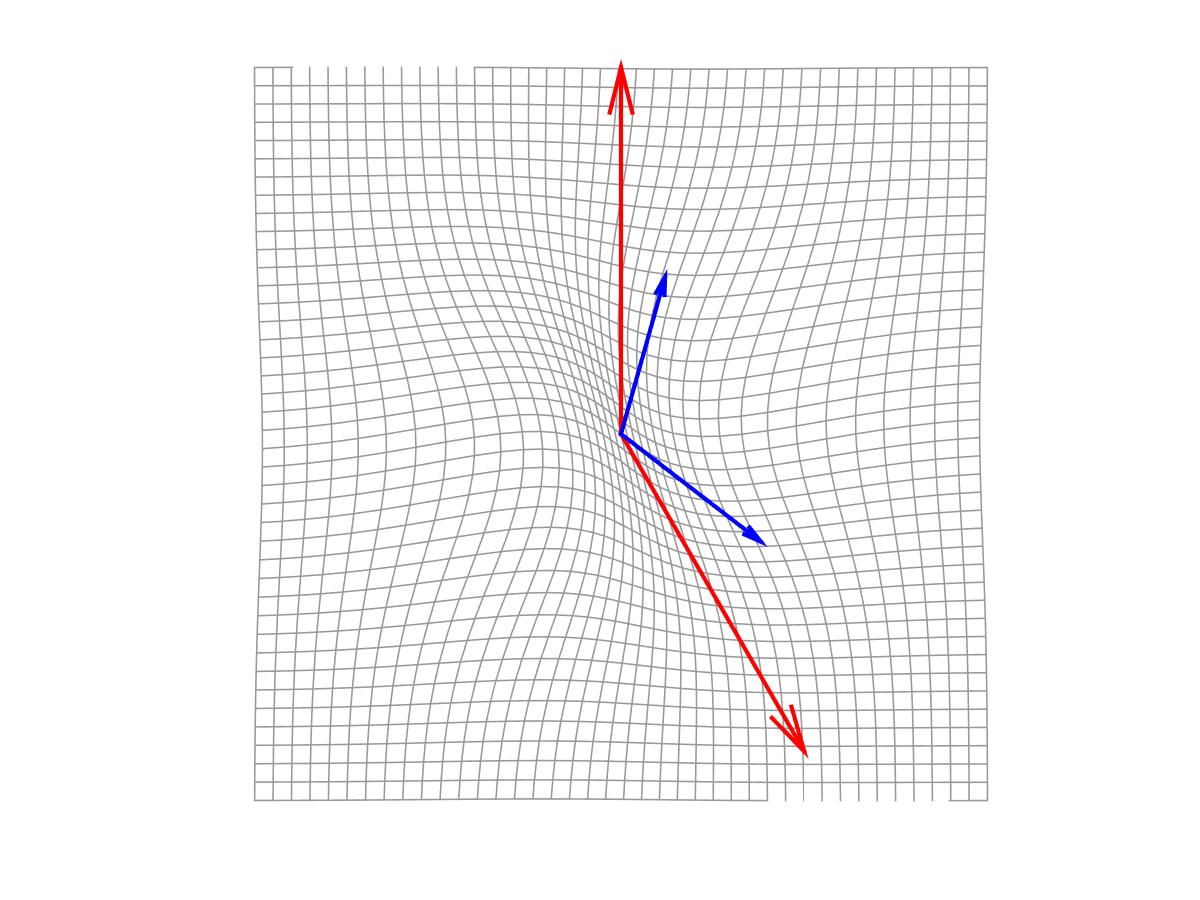} \\
    \rotatebox{90}{\phantom{aa} Or. Gaussian}
    &\includegraphics[trim = 8mm 8mm 8mm 8mm ,clip,width=3cm]{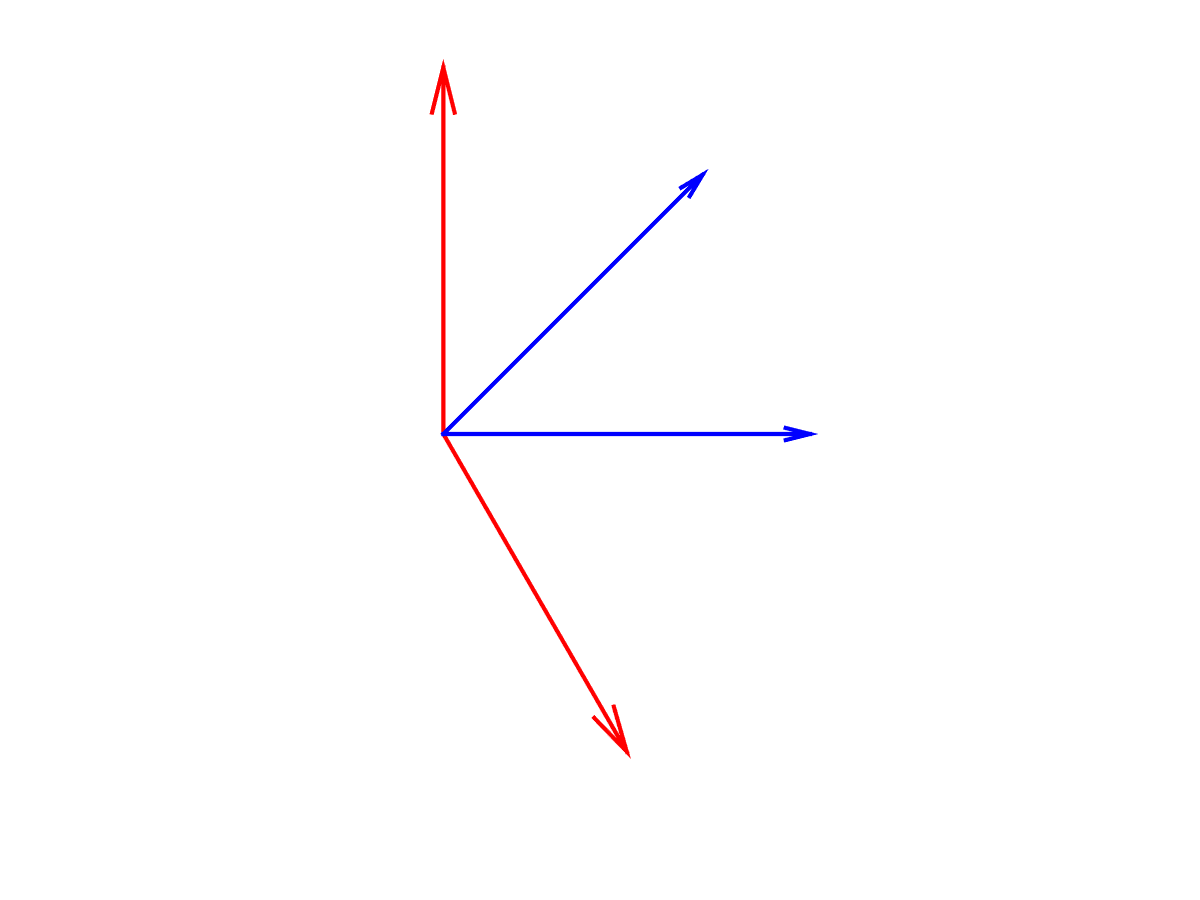} 
    &\includegraphics[trim = 8mm 8mm 8mm 8mm ,clip,width=3cm]{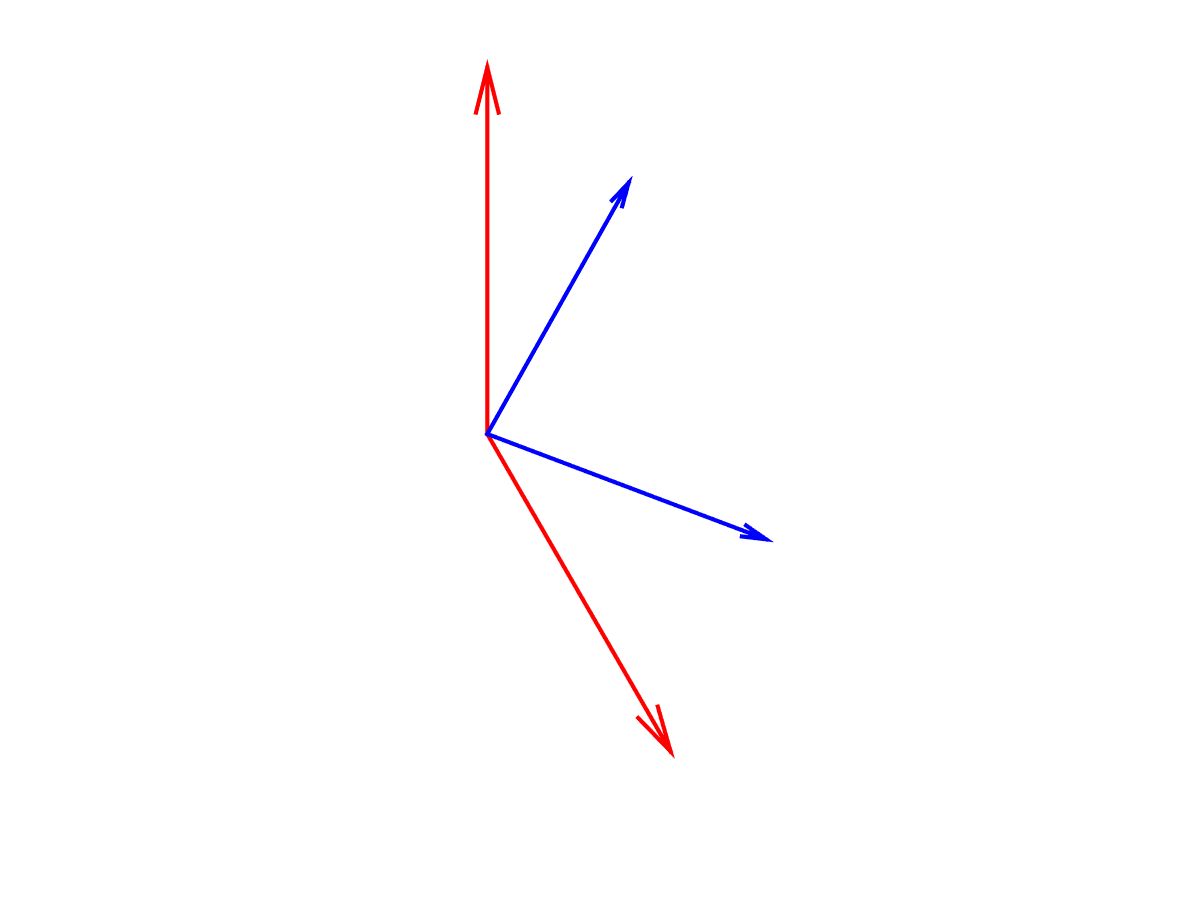} 
    &\includegraphics[trim = 8mm 8mm 8mm 8mm,clip,width=3cm]{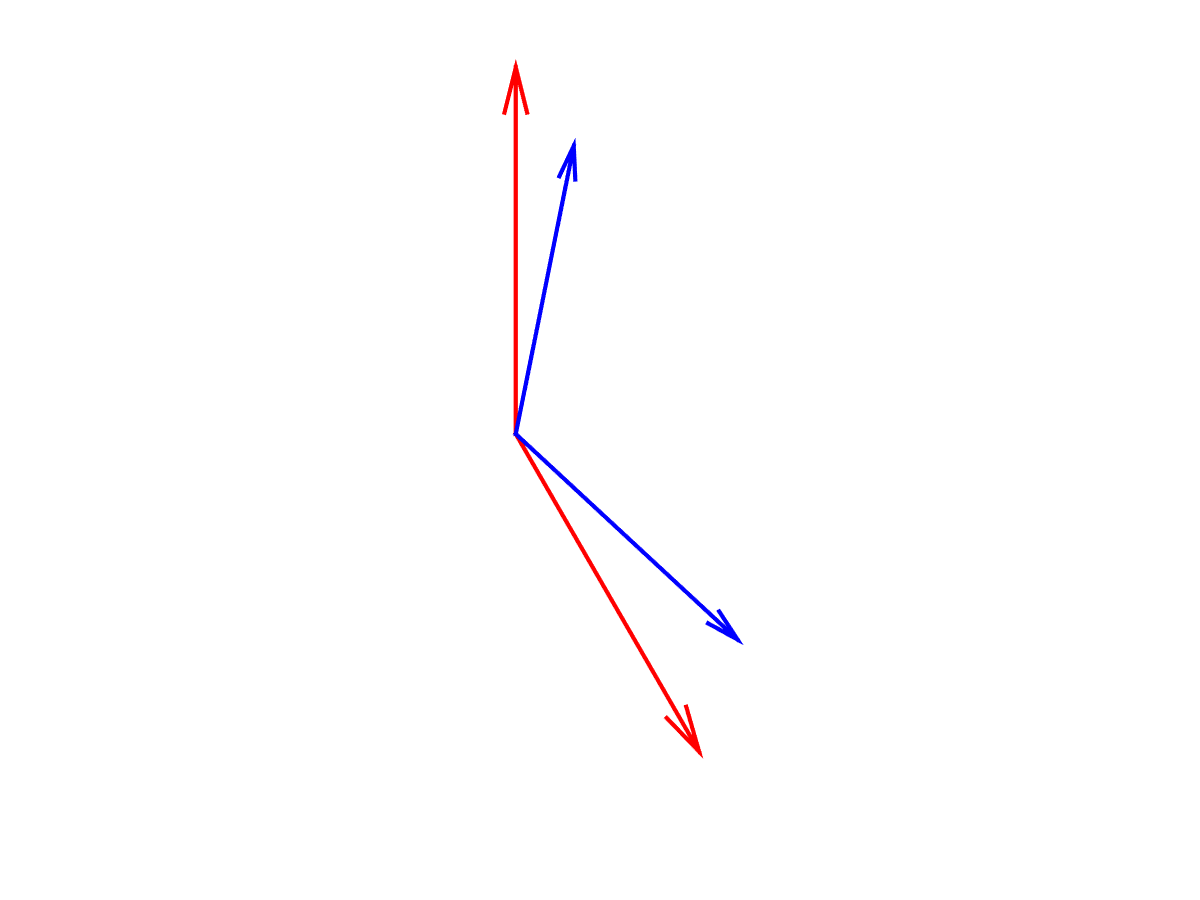}
    &\includegraphics[trim = 8mm 8mm 8mm 8mm ,clip,width=3cm]{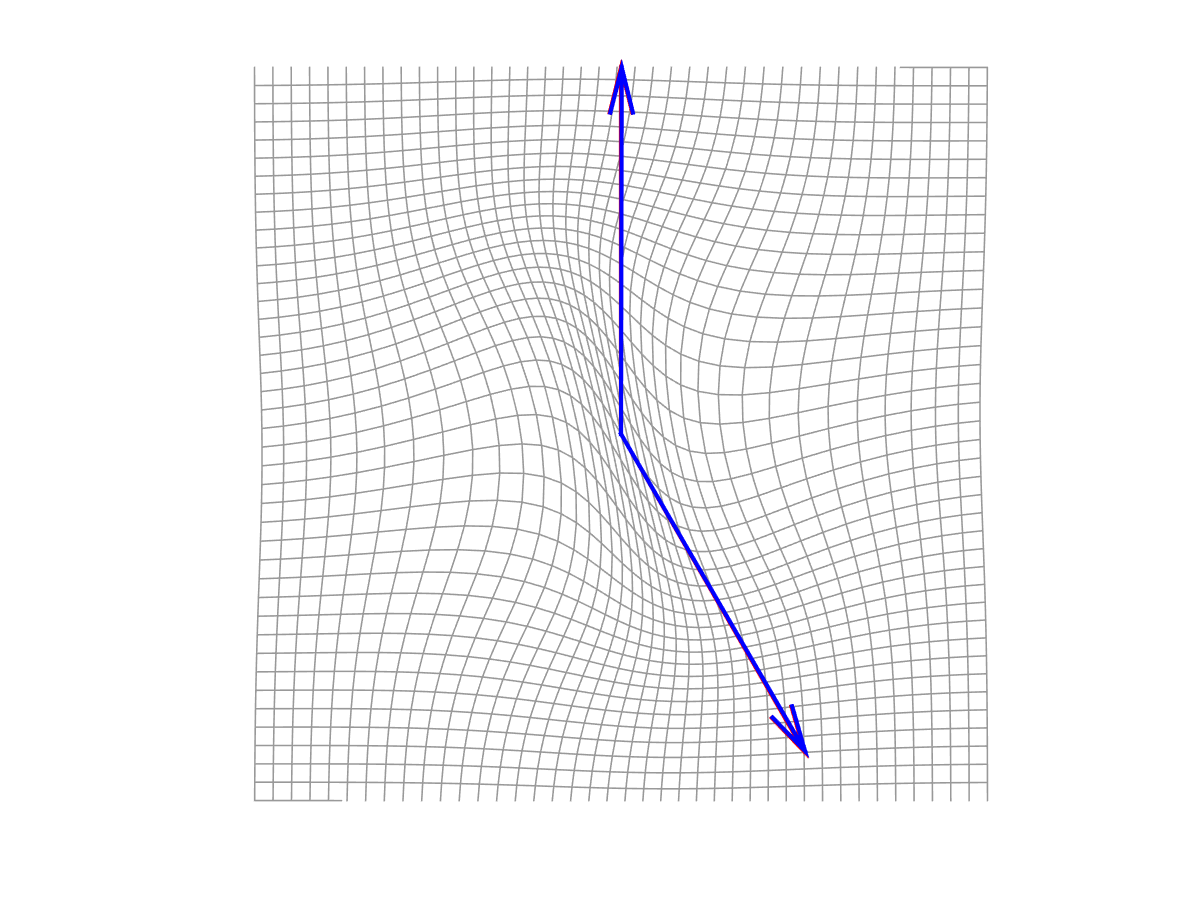} \\
     &$t=0$ & $t=1/3$ & $t=2/3$  & $t=1$
    \end{tabular}
    \caption{Registration of pairs of Dirac varifolds with the pushforward model for both the linear and oriented Gaussian kernel.}
    \label{fig:2Diracs_bis}
\end{figure}

We illustrate the aforementioned properties on a very simple registration example between pairs of Dirac varifolds located at the same position $x$ i.e $\delta_{(x,d_1)} + \delta_{(x,d_2)}$ and $\delta_{(x,d_1')} + \delta_{(x,d_2')}$. In Figure \ref{fig:2Diracs}, the template and target pairs of Diracs are matched based on the normalized action model. The estimated matching and deformations clearly differ with the choice of kernel but each of these result is in fact perfectly consistent with the different invariances of those kernels. Indeed the two Diracs are exactly matched to the target using the oriented Gaussian kernel since $\|\cdot\|_{W^*}$ is in that case a metric on the entire space $\mathcal{D}$. They are however matched to the opposite vectors with the unoriented Gaussian kernel which is indeed insensitive to orientation. In the case of Binet kernel, in addition to orientation-invariance, there exists other pairs of Diracs which are distinct in $\mathcal{D}$ but coincide in $W^*$. For example, it can be easily verified that all discrete varifolds of the form $\delta_{(x,d_1)} + \delta_{(x,d_2)}$ with orthogonal vectors $d_1$ and $d_2$ are equal in $W^*$, which is reflected by the result in Figure \ref{fig:2Diracs}. 

We emphasize the difference of behavior between linear and oriented Gaussian kernels with the example of Figure \ref{fig:2Diracs_bis} associated this time to the pushforward action model. The result shown in the first row is a consequence of the fact that fidelity terms derived from the linear kernel only constrains the sums $d_1 + d_2$ and $d_1' + d_2'$ to match.

\subsubsection*{Multi-directional varifold matching}
Finally, Figure \ref{fig:cat} shows an example of matching on more general discrete varifolds that involve varying number of directions at different spatial locations. This is computed with the normalized action using an oriented Gaussian kernel for the fidelity term. Although purely synthetic, it illustrates the potentialities of the proposed approach to register data with complex directional patterns.    

\begin{figure}
    \centering
    \begin{tabular}{cccc}
    \includegraphics[trim = 15mm 15mm 15mm 15mm ,clip,width=3.2cm]{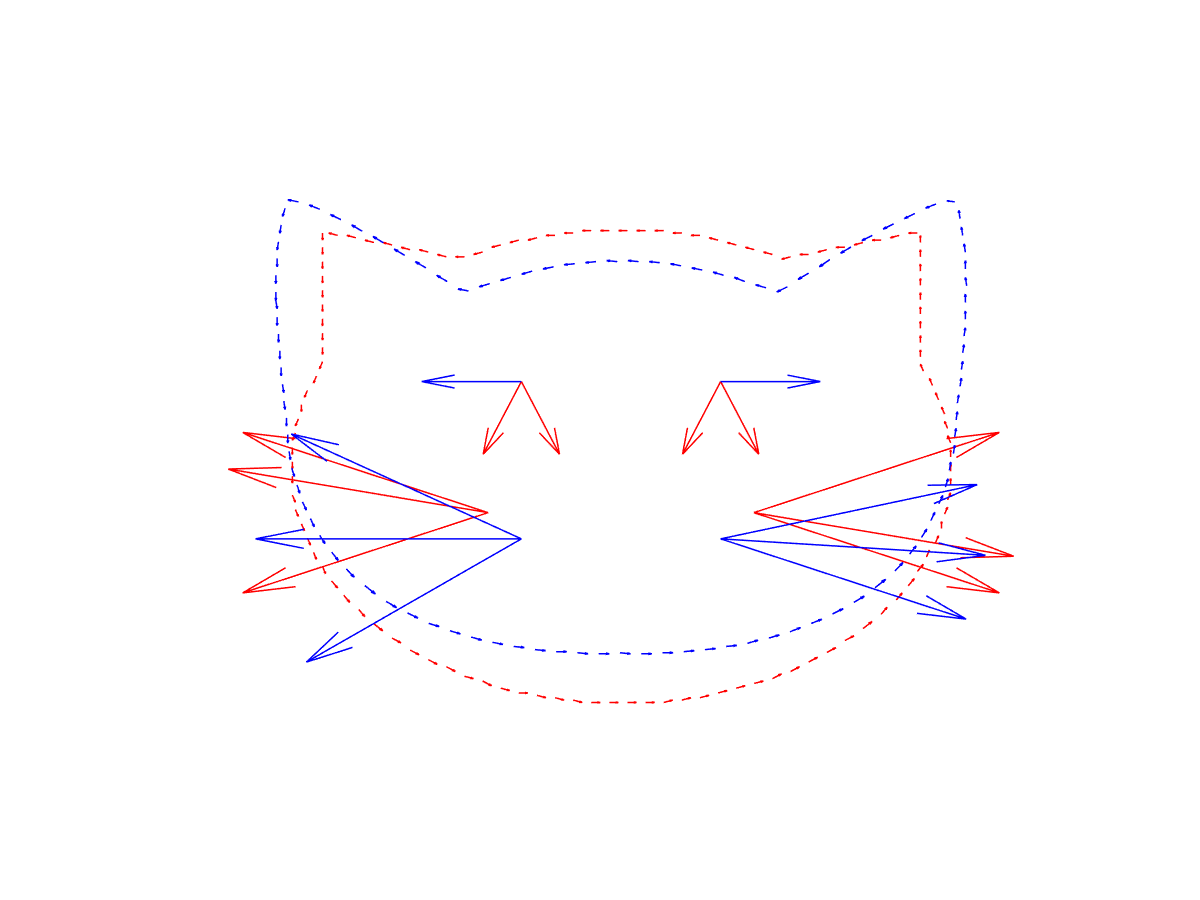} 
    &\includegraphics[trim = 15mm 15mm 15mm 15mm ,clip,width=3.2cm]{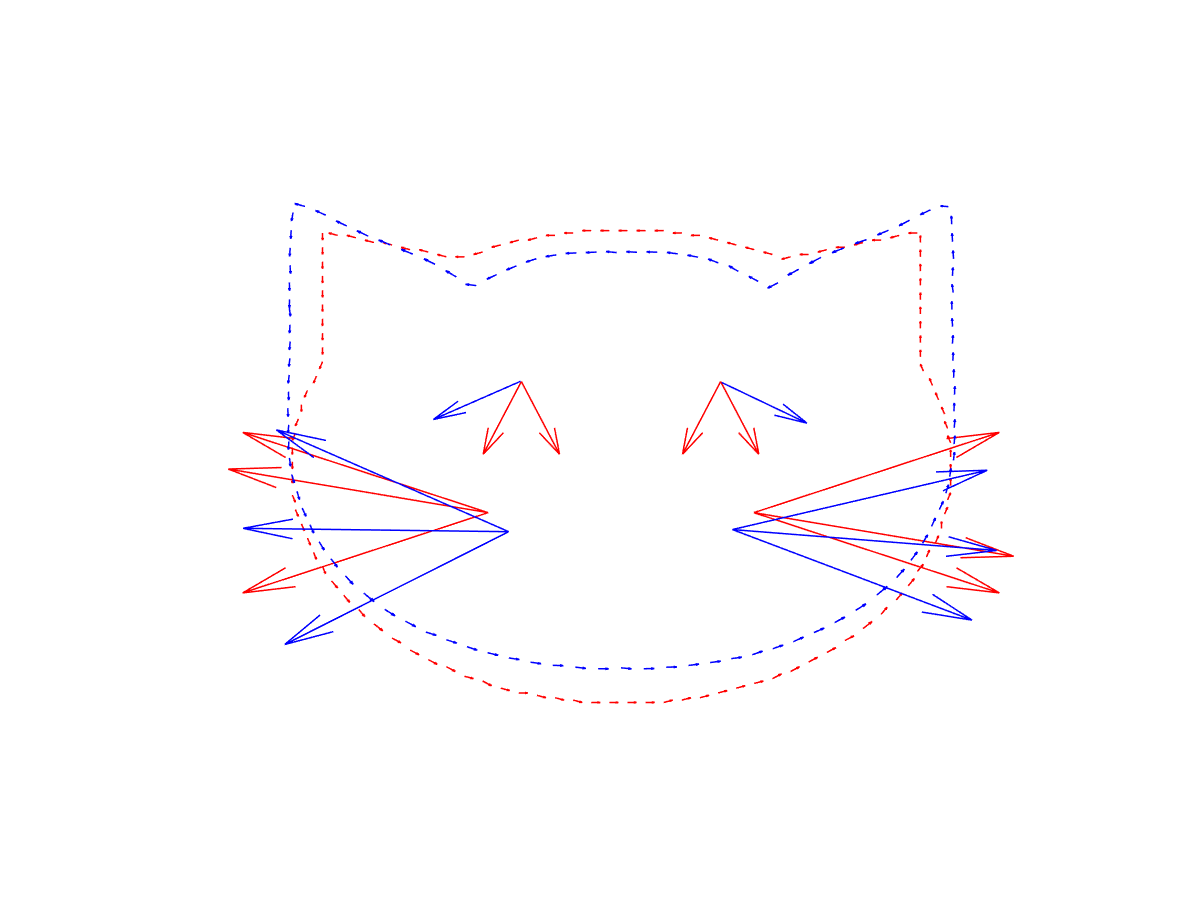} 
    &\includegraphics[trim = 15mm 15mm 15mm 15mm ,clip,width=3.2cm]{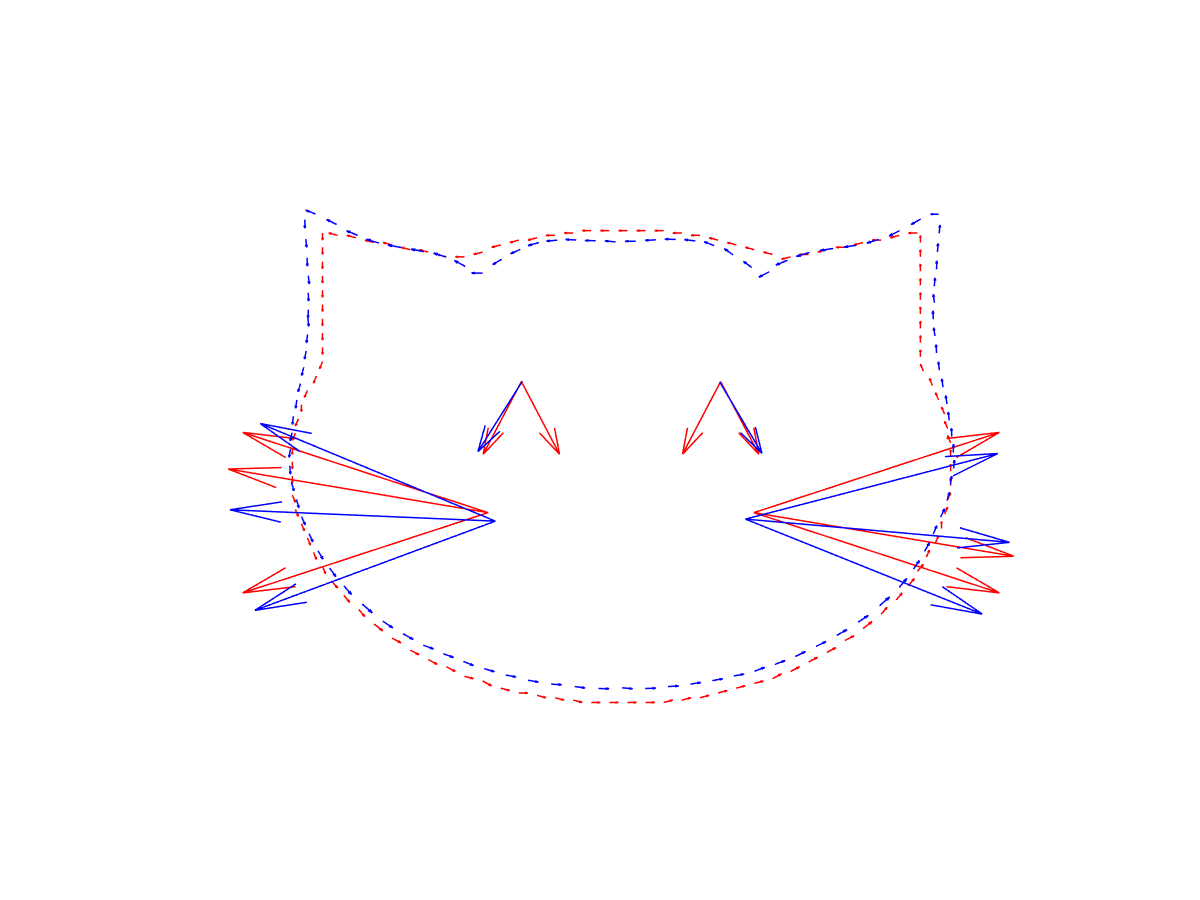}
    &\includegraphics[trim = 15mm 15mm 15mm 15mm ,clip,width=3.2cm]{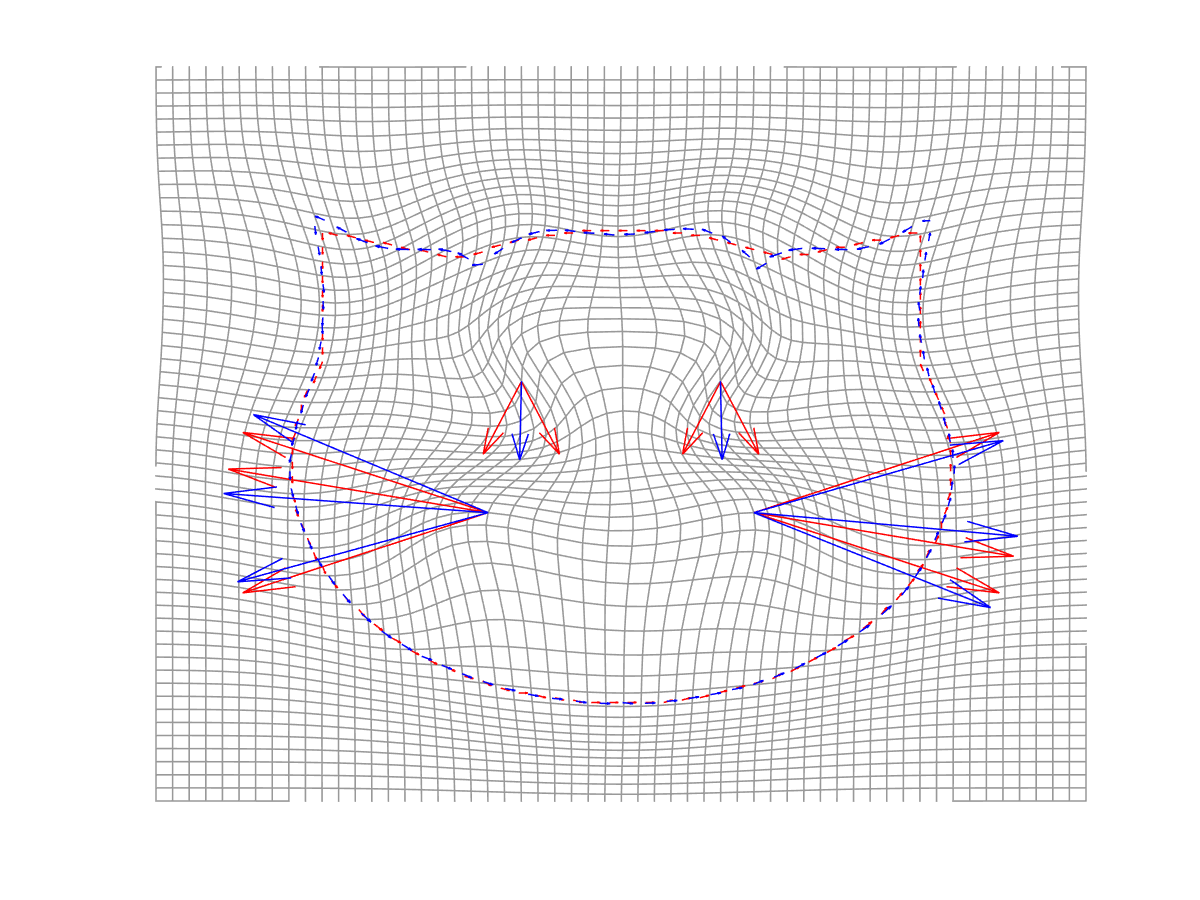} \\
    $t=0$ & $t=1/3$ & $t=2/3$  & $t=1$ 
    \end{tabular}
    \caption{Registration of multi-directional sets. The lengths of vectors correspond to the weights of the Dirac varifolds.}
    \label{fig:cat}
\end{figure}

\subsection{Contrast-invariant image registration}
A last possible application worth mentioning is the registration of images with varying contrast. Indeed, an image $I$ modulo all contrast changes is equivalently represented by its unit gradient vector field $\frac{\nabla I}{|\nabla I|}$. Note that this may in fact be only defined at isolated pixels in the image, specifically the ones where the gradient is non vanishing. Within the setting of this work, it is thus natural to associate to $I$ the discrete varifold 
\begin{equation*}
  \mu_{I} = \sum_{\nabla I(x_i) \neq 0} \delta_{\left(x_i,\frac{\nabla I}{|\nabla I|}(x_i)\right)}   \in \mathcal{D}
\end{equation*}
    

\begin{figure}
    \centering
    \begin{tabular}{cc}
    \includegraphics[trim = 15mm 15mm 15mm 10mm ,clip,width=4cm]{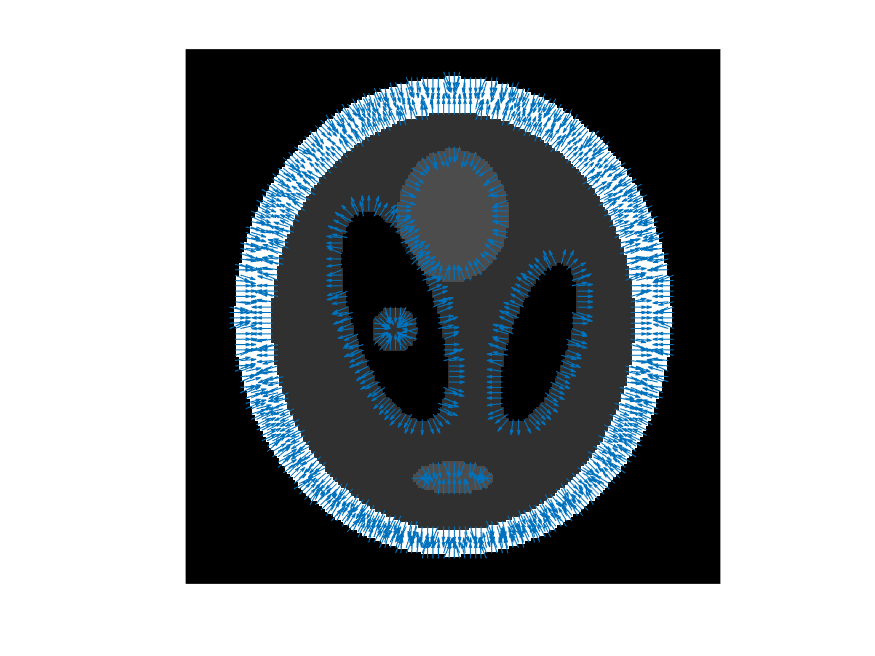} 
    &\includegraphics[trim = 15mm 15mm 15mm 10mm ,clip,width=4cm]{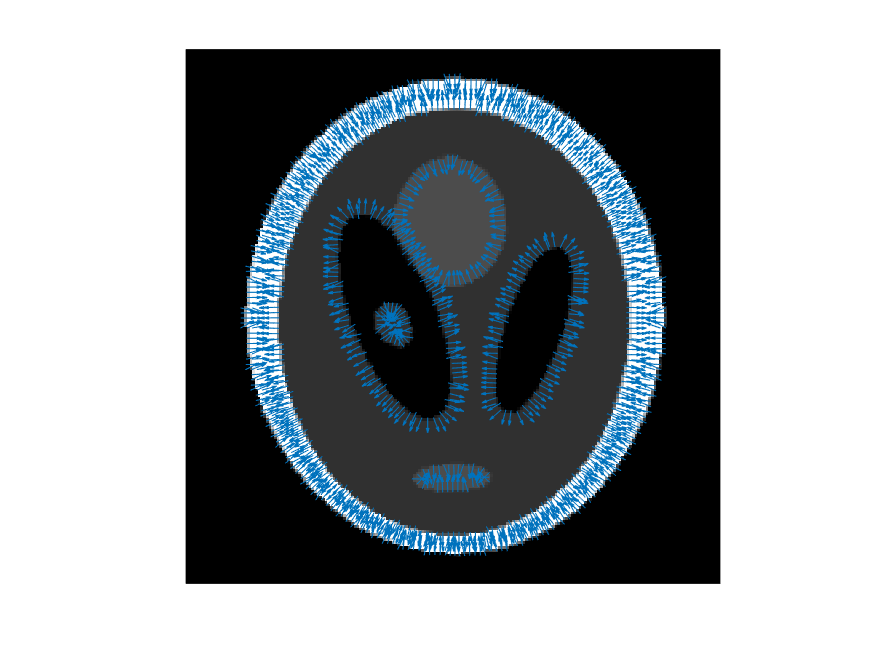} \\
    $t=0$ & $t=1/2$ \\
    \includegraphics[trim = 15mm 15mm 15mm 10mm ,clip,width=4cm]{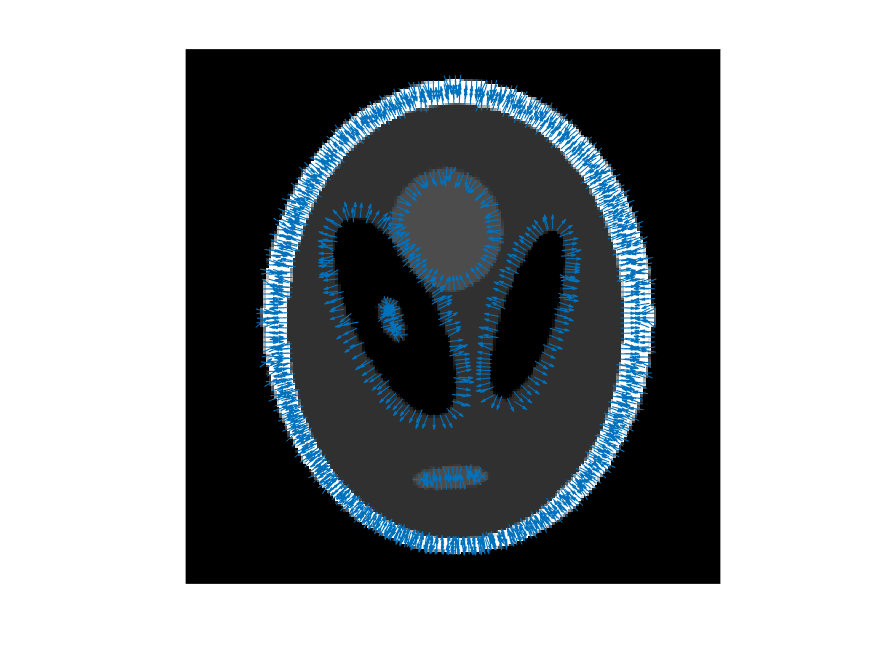}
    &\includegraphics[trim = 15mm 15mm 15mm 10mm ,clip,width=4cm]{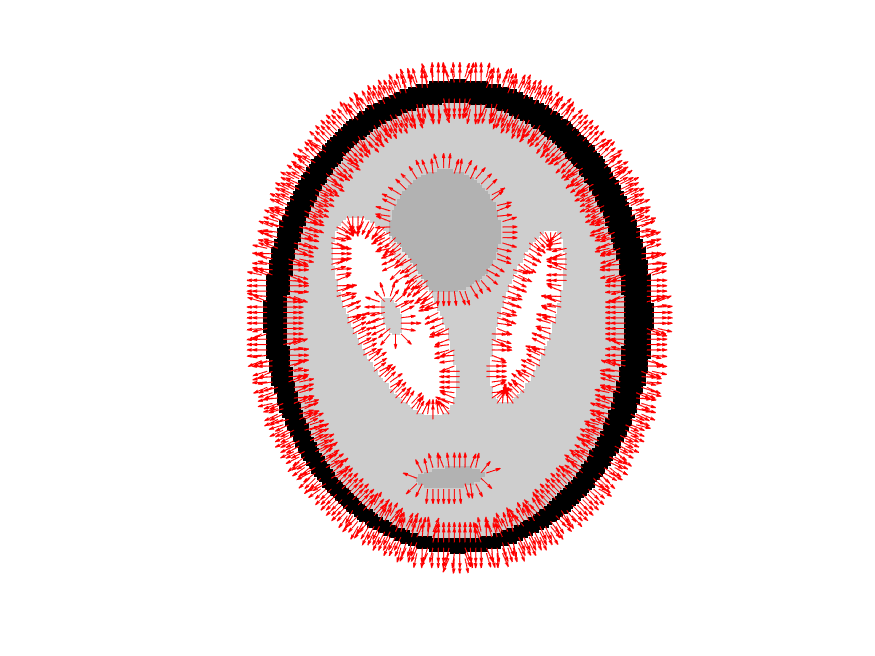} \\
     $t=1$  & $\textrm{target}$ 
    \end{tabular}
    
    \begin{tabular}{c}
    \includegraphics[trim = 15mm 15mm 15mm 10mm ,clip,width=6cm]{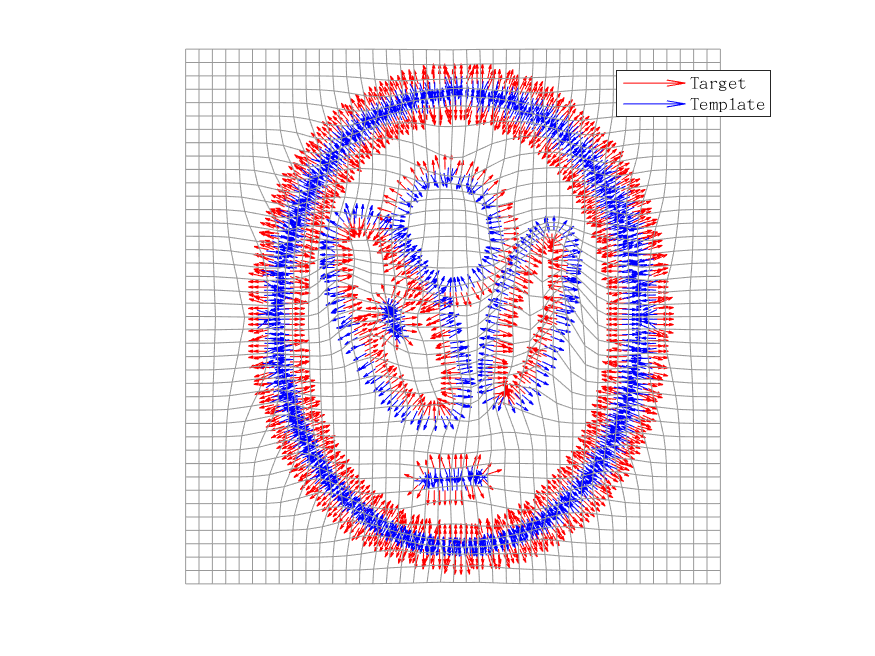} \\
    t=1
    \end{tabular}
    \caption{Registration of images modulo contrast changes. The matching is computed between the discrete varifolds associated to both images with the normalized action model and unoriented Gaussian fidelity term. The estimated deformation can be then applied to the template image.}
    \label{fig:phantom}
\end{figure}

It is straightforward that $\mu_{I}$ is invariant to increasing contrast changes. It also becomes invariant to decreasing ones by quotienting out the orientation of the unit gradient vectors, which in our framework is simply done by selecting an orientation-invariant kernel $\gamma(\langle d ,d' \rangle)$ to define $\|\cdot\|_{W^*}$. In Figure \ref{fig:phantom}, this approach is used to map two oppositely contrasted synthetic phantom brain images. We show both the alignment of the discrete varifolds as well as the full deformation applied to the image itself. Note that these images have no noise and a simple structure with relatively low number of non-vanishing gradients. There will be clearly the need for more validation to be done in the future in order to evaluate the practicality and robustness of this method for real multi-modal medical images.

\section{Conclusion and future work} 
We have proposed, in this paper, a framework for large deformation inexact registration between discrete varifolds. It relies on the LDDMM setting for diffeomorphisms and include different models of group action on the space of varifolds. In each case, we derived the corresponding optimal control problems and the associated geodesic equations in Hamiltonian form. By combining those with the use of kernel-based fidelity metrics on varifolds, we proposed a geodesic shooting algorithm to numerically tackle the optimization problems. We finally illustrated the versatility and properties of this approach through examples of various natures which go beyond the classical cases of curves or surfaces. 

Several improvements or extensions of this work could be considered for future work. From a theoretical standpoint, it would be for instance important to derive a more general 'continuous' varifold matching model i.e with more general distributions than Dirac sums. Besides, higher dimensional varifolds could be possibly introduced within our model, although this would involve dealing with direction elements in Grassmann manifolds as in \cite{Charon2} instead of the simpler $\S^{n-1}$. Lastly, future work will also include adapting the existing fast GPU implementations for LDDMM to the new dynamical systems appearing here, with the objective of making the whole approach more scalable to real data applications.        

\subsection*{Acknowledgements}
The authors would like to thank Prof. Sarang Joshi for many enriching discussions that initiated parts of this work.

\bibliographystyle{amsplain}
\bibliography{biblio}

\section*{\uppercase{Appendix}}
\subsection*{Proof of Proposition \ref{reduced_prop}}
Let $\left(X_i(t),u_i(t),P_i^{(1)}(t),P_i^{(2)}(t)\right)$ and $V_t(\cdot)$ satisfy equations \eqref{reduced_foward} and \eqref{reduced_vf}, then it's straightforward to verify that 
\begin{align*}
&\left( x_i(t),d_i(t),r_i(t),p_i^{(1)}(t),p_i^{(2)}(t),p_i^{(3)}(t) \right) \\
&:= \left( X_i(t),\overline{u_i}(t),|u_i(t)|,P_i^{(1)}(t),|u_i(t)|P_i^{(2)}(t),\left\langle P_i^{(2)}(t),\overline{u_i}(t)\right\rangle \right)
\end{align*}
is a solution of equation \eqref{fullforward} for $V$ with the initial conditions $$(X_i(0),\overline{u_i}(0),|u_i(0)|,P_i^{(1)}(0),|u_i(0)|P_i^{(2)}(0),\langle P_i^{(2)}(0),\overline{u_i}(0)\rangle).$$ Moreover, we see that
\begin{align}\label{eq:cond}
 \left\langle p_i^{(2)}(t), d_i(t) \right\rangle = \left\langle u_i(t),P_i^{(2)}(t) \right\rangle = r_i(t) p_i^{(3)}(t) , \ \forall t
\end{align}
which leads to $V_t$ being equal to the vector field $v_t$ defined in \eqref{fullvf} and therefore to a solution for the system \eqref{fullforward}.
 
Conversely, let $\left(x_i(t),d_i(t),r_i(t),p_i^{(1)}(t),p_i^{(2)}(t),p_i^{(3)}(t)\right))$ and $v_t(\cdot)$ satisfying \eqref{fullforward} and \eqref{fullvf} with initial conditions such that
  \begin{align} \label{eq:ini_cond}
  \left\langle p_i^{(2)}(0), d_i(0) \right\rangle = r_i(0) p_i^{(3)}(0). 
  \end{align}
Now let $\left(X_i(t),u_i(t),P_i^{(1)}(t),P_i^{(2)}(t)\right)$ be the solution of \eqref{reduced_foward} with the initial condition 
$$\left(x_i(0),r_i(0) d_i(0), p_i^{(1)}(0),p_i^{(2)}(0) \right)$$ 
and vector field $v_t(\cdot)$. We define $V_t(\cdot)$ as in \eqref{reduced_vf}, then as in previous discussion, we see that

\begin{align*}
\left( X_i(t),\overline{u_i}(t),|u_i(t)|,P_i^{(1)}(t),|u_i(t)|P_i^{(2)}(t),\left\langle P_i^{(2)}(t),\overline{u_i}(t) \right\rangle \right)
\end{align*}
is the solution for \eqref{fullforward} with initial value
\begin{align*}
&\left( X_i(0),\overline{u_i}(0),|u_i(0)|,P_i^{(1)}(0),|u_i(0)|P_i^{(2)}(0),\left\langle P_i^{(2)}(0),\overline{u_i}(0) \right\rangle \right) \\
 &= \left(x_i(0),d_i(0),r_i(0),p_i^{(1)}(0),p_i^{(2)}(0),p_i^{(3)}(0) \right).
\end{align*}
Since $\left(x_i(t),d_i(t),r_i(t),p_i^{(1)}(t),p_i^{(2)}(t),p_i^{(3)}(t)\right))$ is a solution for the same initial value problem, by uniqueness of ODE, we obtain

\begin{align*}
&\left( x_i(t),d_i(t),r_i(t),p_i^{(1)}(t),p_i^{(2)}(t),p_i^{(3)}(t) \right) \\
&= \left( X_i(t),\overline{u_i}(t),|u_i(t)|,P_i^{(1)}(t),|u_i(t)|P_i^{(2)}(t),\left\langle P_i^{(2)}(t),\overline{u_i}(t) \right\rangle \right), \ \forall t \in [0,1].
\end{align*}
Also, we have equation \eqref{eq:cond}, and from this equation we have
$$P_{d_k^{\perp}}(p_k^{(2)}) + p_k^{(3)} r_k d_k = p_k^{(2)} + \left(\left\langle p_k^{(2)}, d_k \right\rangle - p_k^{(3)} r_k \right) d_k  = p_k^{(2)} $$
and hence
\begin{align*}
v_t(\cdot) &= \sum_{k=1}^P K(x_k,\cdot)p_k^{(1)}+ D_1K(x_k,\cdot)(d_k,p_k^{(2)}) \\
 &= \sum_{k=1}^P K(X_k,\cdot)P_k^{(1)}+ D_1K(X_k,\cdot)(u_k,P_k^{(2)}) =V_t(\cdot).
\end{align*}

\subsection*{Reduced Hamiltonian equations}
For convenience, let us denote $f(|x_k-x_i|^2)$ by $f_{ki}$ for any function $f$. Then the reduced Hamiltonian equations for the normalized action can be shown to be
\begin{align*}
\left\{\begin{array}{ll}
 \dot{x}_i & = \sum_{k=1}^P h_{ki} p_k^{(1)} +2 \dot h_{ki}\langle x_k-x_i, d_k\rangle P_{d_k^{\perp}}(p_k^{(2)}) \\ 
 \dot{d}_i &= \sum_{k=1}^P -2 \dot{h}_{ki} \langle x_k-x_i,d_i\rangle P_{d_k^{\perp}}\left(p_k^{(1)}\right) \\
 &-\left[ 4 \ddot{h}_{ki} \langle x_k-x_i,d_k \rangle \langle x_k-x_i,d_i \rangle +2 \dot{h}_{ki} \langle d_k,d_i\rangle\right] P_{d_k^{\perp}}(p_k^{(2)}) \\
\dot{p}_i^{(1)} &= \sum_{k=1}^P \bigg\{ \left[ 2 \dot{h}_{ki} \langle p_k^{(1)}, p_i^{(1)} \rangle +4 \ddot{h}_{ki} \langle x_k-x_i, d_k \rangle \langle P_{d_k^{\perp}} (p_k^{(2)}), p_i^{(1)} \rangle \right] \\
&- \bigg[ 4 \ddot{h}_{ki} \langle x_k-x_i, d_i \rangle \langle p_k^{(1)}, P_{d_i^{\perp}} (p_i^{(2)}) \rangle \\
 &+ \bigg(8 h^{(3)}_{ki} \langle x_k-x_i, d_k \rangle \langle x_k-x_i, d_i \rangle+ 4 \ddot{h}_{ki} \langle d_k, d_i \rangle \bigg) \langle  P_{d_k^{\perp}}(p_k^{(2)}), P_{d_i^{\perp}}(p_i^{(2)})\rangle \bigg] \bigg\} (x_k-x_i) \\
&+ \left[ 2 \dot{h}_{ki} \langle P_{d_k^{\perp}}(p_k^{(2)}),p_i^{(1)} \rangle- 4 \ddot{h}_{ki} \langle x_k-x_i, d_i \rangle \langle  P_{d_k^{\perp}}(p_k^{(2)}), P_{d_i^{\perp}}(p_i^{(2)}) \rangle \right]  d_k \\
&- \left[ 2 \dot{h}_{ki} \langle p_k^{(1)}, P_{d_i^{\perp}}(p_i^{(2)}) \rangle + 4 \ddot{h}_{ki} \langle x_k-x_i, d_k \rangle \langle  P_{d_k^{\perp}}(p_k^{(2)}), P_{d_i^{\perp}}(p_i^{(2)}) \rangle \right] d_i \\
\dot{p}_i^{(2)} &= \sum_{i=1}^P \left[2 \dot{h}_{ki} \langle p_k^{(1)}, P_{d_i^{\perp}} (p_i^{(2)}) \rangle + 4 \ddot{h}_{ki} \langle x_k-x_i,d_k \rangle \left\langle P_{d_k^{\perp}}(p_k^{(2)}), P_{d_i^{\perp}} (p_i^{(2)})\right\rangle \right] (x_k-x_i) \\
&+ 2 \dot{h}_{ki} \left\langle P_{d_k^{\perp}}(p_k^{(2)}), P_{d_i^{\perp}} (p_i^{(2)})\right\rangle d_k  \\
&-\left\langle d_i,p_i^{(2)}\right\rangle \Big\{ 2 \dot{h}_{ki} \langle x_k-x_i,d_i \rangle p_k^{(1)} + \Big[4 \ddot{h}_{ki} \langle x_k-x_i, d_k\rangle \langle x_k-x_i,d_i \rangle \\
&+ 2 \dot{h}_{ki} \langle d_k,d_i\rangle  \Big] P_{d_k^{\perp}}(p_k^{(2)}) \Big\} \\
&-\Big\{ 2 \dot{h}_{ki} \langle x_k-x_i,d_i\rangle \langle p_k^{(1)},d_i \rangle \\
&+ \Big[4 \ddot{h}_{ki} \langle x_k-x_i,d_k \rangle \langle x_k-x_i,d_i\rangle + 2 \dot{h}_{ki} \langle d_k,d_i\rangle \Big] \langle P_{d_k^{\perp}}(p_k^{(2)}),d_i \rangle \Big\} p_i^{(2)}
\end{array}\right.
\end{align*}

In the pushforward action case, these equations are:

\begin{align*}
\left\{\begin{array}{ll}
\dot x_i &= \sum_{k=1}^P  h_{ki} p_k^{(1)} + 2 \dot h_{ki} \langle x_k-x_i, u_k\rangle p_k^{(2)} \\
\dot u_i &= \sum_{k=1}^P -2 \dot{h}_{ki} \langle x_k-x_i, u_i \rangle p_k^{(1)}- \left[4 \ddot{h}_{ki} \langle x_k-x_i, u_i \rangle \langle x_k-x_i,u_k \rangle + 2 \dot{h}_{ki} \langle u_i, u_k \rangle \right] p_k^{(2)}\\
\dot p_i^{(1)} &= \sum_{k=1}^P \Big\{\left[ 2 \dot{h}_{ki} \left\langle p_k^{(1)}, p_i^{(1)} \right\rangle + 4 \ddot{h}_{ki} \langle x_k-x_i, u_k \rangle \left\langle p_k^{(2)}, p_i^{(1)} \right\rangle \right] \\
&- \Big[4 \ddot{h}_{ki} \langle x_k-x_i, u_i\rangle \langle p_k^{(1)}, p_i^{(2)} \rangle \\
&+\left(8 h_{ki}^{(3)} \langle x_k-x_i,u_k \rangle \langle x_k-x_i, u_i \rangle +4 \ddot{h}_{ki} \langle d_k,d_i \rangle \right) \langle p_k^{(2)},p_i^{(2)}\rangle \Big] \Big\}(x_k-x_i) \\
&+\left[ 2 \dot{h}_{ki} \langle p_k^{(2)}, p_i^{(1)}\rangle - 4 \ddot{h}_{ki} \langle x_k-x_i, u_i \rangle \langle p_k^{(2)}, p_i^{(2)} \rangle \right]u_k\\
&-\Big[2 \dot{h}_{ki} \langle p_k^{(1)}, p_i^{(2)} \rangle +4 \ddot{h}_{ki} \langle x_k-x_i,u_k \rangle \langle  p_k^{(2)}, p_i^{(2)} \rangle \Big] u_i \\
\dot p_i^{(2)}  &=  \sum_{k=1}^P \left[ 2 \dot{h}_{ki} \left\langle p_k^{(1)}, p_i^{(2)} \right\rangle+ 4 \ddot{h}_{ki} \langle x_k-x_i,u_k\rangle \left\langle p_k^{(2)}, p_i^{(2)} \right\rangle \right](x_k-x_i) \\
&+ 2 \dot{h}_{ki} \left\langle p_k^{(2)} ,p_i^{(2)} \right\rangle u_k
\end{array}\right.
\end{align*}

\end{document}